\documentclass[11pt,reqno]{amsart}

\usepackage{a4wide}

\newif\ifproofs \newif\iftodos \newif\ifrefchecks

\proofstrue
\todostrue
\refchecksfalse

\usepackage{eqparbox}

\usepackage{amsmath}
\usepackage{amssymb}
\usepackage{graphicx}
\usepackage{mathtools}
\usepackage{enumerate}
\usepackage{comment}

\newtheorem{thm}{Theorem}[section]
\newtheorem*{thm*}{Theorem}
\newtheorem{prop}[thm]{Proposition}

\newtheorem{lem}[thm]{Lemma}
\newtheorem*{lem*}{Lemma}
\newtheorem{cor}[thm]{Corollary}
\newtheorem{defn}[thm]{Definition}
\theoremstyle{remark}
\newtheorem{rema}[thm]{\bf Remark}
\newtheorem{remas}[thm]{\bf Remarks}

\numberwithin{equation}{section}

\usepackage[color=green]{todonotes}

\def\wh{\widehat}
\def\mbb{\mathbb}
\def\mb{\mathbf}
\def\mf{\mathfrak}

\def\mc{\mathcal}

\newlength{\equwidth}
\DeclareMathOperator{\Proj}{Proj}
\DeclareMathOperator{\Ric}{Ric}
\DeclareMathOperator{\Rho}{P}

\def\B{\mathcal B}

\def\al{\alpha}
\def\be{\beta}
\def\ga{\gamma}
\def\de{\delta}

\def\si{\sigma}
\def\ta{\tau}
\def\ups{\upsilon}
\def\Ups{\Upsilon}
\def\ph{\varphi}

\def\ps{\psi}
\def\om{\omega}
\def\Ga{\Gamma}

\def\Th{\Theta}

\def\Ph{\Phi}
\def\Ps{\Psi}

\def\pa{\partial}

\def\idx#1{{\em #1\/}}
\def\del{\pa}

\def\im{\mathrm{im}\ }
\def\dim{\mathrm{dim\ }}

\def\wt{\widetilde}
\def\ol{\overline}

\def\O{\ensuremath{\mathrm{O}}}

\def\g{\ensuremath{\mathfrak{g}}}

\def\ce{\ensuremath{\mathcal{E}}}

\newcommand{\bep}{\mbox{\boldmath{$ \epsilon$}}}

\def\today{\ifcase\month\or
 January\or February\or March\or April\or May\or June\or
 July\or August\or September\or October\or November\or December\fi
 \space\number\day, \number\year}

\usepackage{paralist} 
\usepackage{graphicx} 
\usepackage{hyperref}
\usepackage[all]{xy}

\def\mb{\mathbf} 
\def\mf{\mathfrak}

\def\mc{\mathcal}

\def\d{\operatorname{d}\!}

\usepackage{tensor}
\newcommand{\ind}{\indices}
\newcommand{\parderv}[2] {\frac{\partial#1}{\partial#2}}

\newenvironment{note}{\bgroup \par \medskip \color{magenta} \ttfamily}{\rmfamily \medskip \egroup}

\ifproofs
\else
\todosfalse
\refchecksfalse
\usepackage{environ}
\NewEnviron{killcontents}{}

\fi

\iftodos\else\presetkeys{todonotes}{disable}{}\fi

\ifrefchecks\usepackage{refcheck}\fi

\title{Modified conformal extensions}

\author[Hammerl, Sagerschnig, \v{S}ilhan, \v{Z}\'adn\'ik]{Matthias Hammerl, Katja Sagerschnig,
Josef \v{S}ilhan and Vojt\v{e}ch \v{Z}\'adn\'ik}

\address{\flushleft M. H.:  University of Applied Sciences Wiener Neustadt, Austria
\newline K. S.: Center for Theoretical Physics PAS, Al. Lotników 32/46, 02-668 Warszawa, Poland
\newline J. \v S.: Masaryk University, Faculty of Science, Kotl\'{a}\v{r}sk\'{a} 2, 61137 Brno, Czech Republic
\newline 
V. \v{Z}: Masaryk University, Faculty of Education, Po\v{r}\'\i\v{c}\'\i\ 31, 60300 Brno, Czech Republic
}
\email{matthiasrh@gmail.com, kat.sagerschnig@gmail.com, silhan@math.muni.cz, 
\newline zadnik@mail.muni.cz}

\date{\today}

\subjclass[2000]{53A20, 53A30, 53B30, 53C07} 

\keywords{Differential geometry, Patterson--Walker metric, Projective structure, Conformal structure, 
Conformal Killing field, Einstein metric, Fefferman--Graham ambient metrics}

\begin{document}

\maketitle

\begin{abstract}
We present a geometric construction and characterization of $2n$-dimensional split-signature conformal structures endowed with  a twistor spinor with integrable kernel. 
The construction is regarded as a modification of the conformal Patterson--Walker metric construction for $n$-dimensional projective manifolds.
The characterization is presented in terms of the twistor spinor and an integrability condition on the conformal Weyl curvature. 
We further derive complete description of Einstein metrics and infinitesimal conformal symmetries in terms of suitable projective data.
Finally, we obtain an explicit geometrically constructed Fefferman--Graham ambient metric and show vanishing of $Q$-curvature.
\end{abstract}

\section{Introduction} \label{Intro}

\bgroup
\renewcommand{\thethm}{\Alph{thm}}

Walker manifolds are pseudo-Riemannian manifolds admitting a parallel isotropic distribution. 
In the case of split-signature, there is an interesting subclass of Walker metrics that are induced on the cotangent bundle of a manifold $M$ with a torsion-free affine connection $D$.
These are the \emph{standard Patterson--Walker metrics} or \emph{(pseudo-)Riemannian extensions} of affine structures, introduced in \cite{Patterson1952}.
A Patterson--Walker (shortly, PW) metric on the total space of $T^*M$ is given by the natural pairing between the vertical distribution of the projection $T^*M\to M$ and the horizontal distribution corresponding to $D$.
It has the following coordinate expression
\begin{align}\label{pattersonwalker}
\wt{g} =  2 \, \d x^A \odot \d p_A -2 \, \Ga\ind{_A^C_B} \, p_C\, \d x^A \odot \d x^B \, ,
\end{align}
where $(x^A)$ are local coordinates on $M$, $(p_A)$ the canonical fibre coordinates and $\Gamma \ind{_A^C_B}$ the Christoffel symbols of the underlying affine connection $D$.
Here, the vertical distribution $V\subset TT^*M$ is the isotropic distribution that is parallel with respect to the Levi-Civita connection of $\wt{g}$.

There are variants or modifications of the standard PW construction that still yield Walker metrics. 
They appear with various motivations and applications in both early and more recent references, see, e.g., \cite{Patterson1952}, \cite{Walker1954}, \cite{Patterson1954}, \cite{Afifi1954}, \cite{Derdzinski2009}, \cite{Gilkey2009a}, \cite{Dunajski2018}, \cite{cap-mettler}.
A tight relation between the underlying and the induced data is a common feature of all such constructions.
In this article, we focus on modifications of the form 
\begin{align}
\ol{g}:=\wt{g}+\pi^*\Phi \, ,
\label{mod-PW}
\end{align}
where $\wt{g}$ is the standard PW metric and $\pi^*\Phi$ is the pullback of a symmetric 2-tensor $\Phi$ on $M$ with respect to the projection $\pi:T^*M\to M$.
We call such metrics the \emph{$\Phi$-modified} or just \emph{modified PW metrics}, noticing that other names can be found in the literature.

Considering the relations between the initial affine connection on $M$ and the induced PW metric on $T^*M$, it is natural to analyze the effect of projective change.
It is observed already in \cite{Afifi1954} that projectively flat affine connections correspond to conformally flat metrics.
A projective-to-conformal adaptation of the standard PW construction is made precise and its various aspects are explored in the series of papers \cite{hsstz-fefferman}, \cite{hsstz-ambient}, \cite{hsstz-walker}.
For a projective change of an affine connection to lead to a conformal change of the induced metric, several specifications are needed.
In particular, $M$ is supposed to be oriented and the cotangent bundle $T^*M$ is replaced by its appropriately weighted variant.
In this case, we speak of the \emph{conformal PW metric} or the \emph{conformal extension} of a projective structure. 
The conformal PW metrics were characterized in conformally invariant terms and relationships between infinitesimal symmetries and other objects on the respective structures were obtained. 
Our aim is to extend this research to the modified case.
In the present article, we study the \emph{modified conformal PW metrics} or the \emph{modified conformal extensions} of projective structures, which are the conformal structures that can be represented by a modified PW metric \eqref{mod-PW}. 
To keep the equivariance of  the construction,  the modification tensor $\Phi$ is also to be weighted.
Details and preliminary observations are in section \ref{Prelim}.

Our first result concerns the characterization of the flatness of the modified conformal extension via the flatness of the initial projective structure and certain condition on the modification tensor, see Theorem \ref{relateFlat}. 
Here we encounter the role of invariant differential operators called \emph{BGG operators}, which appear repeatedly in the entire article.
For the reader's convenience, a quick introduction and details on all projective operators used in the article are collected in appendix~\ref{app-A}.

Conformal extensions of projective structures are characterized in \cite[Theorem 1]{hsstz-walker}.
Strictly speaking, the characterization concerns conformal spin structures.
Since we discuss local pro\-perties throughout this article, we may skip the spin assumption.
Relaxing some of conditions in \cite[Theorem 1]{hsstz-walker}, we obtain the characterization of modified conformal extensions in Theorem \ref{thm-char}:

\begin{thm} 
Locally, a conformal structure of split signature is a modified conformal extension of a projective structure if and only if the following properties are satisfied:
\begin{enumerate}[(a)]
\item there is a pure twistor spinor $\chi$  with integrable kernel $\ker\chi$,
\item the following integrability condition holds
\begin{align*} 
{W}_{ab}{}^c{}_d \, v^a w^d & =0 , \quad \text{for all $v^a, w^d\in\ker\chi$} ,
\end{align*}
where ${W}_{ab}{}^c{}_d$ is the conformal Weyl curvature.
\end{enumerate}
\end{thm}
\noindent
Here, of course, $\ker\chi$ corresponds to the vertical distribution $V$.
The theorem can be understood as a conformal version of known results in the (pseudo-)Riemannian setting, cf. \cite{Afifi1954} and \cite{Derdzinski2009}. 
Another ingredient entering the proof, as well as many other places in this article, is the existence of a metric in the conformal class whose Levi-Civita connection leaves $\chi$ parallel, see \cite[Proposition 4.2]{hsstz-walker}.

The central part of the article deals with special properties and objects on modified conformal extensions and their relations to underlying projective counterparts.
In particular, for a modified PW metric, we provide a complete description of Einstein metrics in its conformal class and the infinitesimal conformal symmetries in section \ref{Einstein} and \ref{Symm}, respectively.
The description is given in terms of geometric data on the underlying projective manifold, 
the main results are collected in Theorems \ref{downaE} and \ref{summaryCKE}. 
Although the technicalities may seem discouraging, the scheme is clear:
both for Einstein metrics and infinitesimal conformal symmetries, there are just few building blocks, satisfying certain projectively invariant conditions, from which the objects of interest are constructed.
In the standard (non-modified) case, the blocks are well separated and the conditions consist mainly of the BGG equations, cf. \cite{hsstz-walker}.
In general, the equations are twisted by additional differential operators involving the modification term $\Phi$.
This is also why the current development is in many respects different from the non-modified situation.

The conditions characterizing Einstein scales and infinitesimal conformal symmetries simplify considerably in certain 
special cases that are discussed in section \ref{aEspecial} and \ref{cKspecial}, respectively.
In such cases, the otherwise intertwined conditions can be (at least partly) separated and simplified.
For the sake of illustration, we pick Theorem \ref{aEprojflat} that describes Einstein metrics of modified conformal 
extensions of projectively flat underlying structures:

\begin{thm} \label{thm-B}
Consider a modified PW metric associated to a projectively flat affine connection $D_A$ and a modification tensor 
$\Ph_{AB}\in\ce_{(AB)}(2)$. 
Locally, there is a bijective correspondence between almost Einstein scales of the conformal class and pairs $\ta \in \ce(1)$ and $\xi^A \in \ce^A(-1)$ satisfying 
\begin{align}
& (D_{A} D_{B} + \Rho_{AB}) \ta = 0 , \label{i-flat-ta} \\
& (D_A \xi^B)_0=0 , \quad \xi^R\, \B_2(\Ph)_{ABCR}=0 . \label{i-flat-xi}
\end{align}
The dimension of the space of almost Einstein scales equals to $d+n+1$, where $d$ is the dimension of the space of solutions to \eqref{i-flat-xi} and $n$ is the dimension of the underlying projective manifold.
\end{thm}
\noindent
Here and below, $\Rho_{AB}$ is the projective Schouten tensor of $D_A$ and the numbers in parentheses are the projective weights.
The operator $\B_2: \ce_{(AB)}(2) \to \ce_{[AB][CD]}(2)$ is the projective second BGG operator in the sequence starting with $\ce_A(2)$.
The summand $n+1$ in the previous count is the dimension of the solution space to \eqref{i-flat-ta}, which is a (first) BGG equation corresponding to so-called almost Ricci-flat scales.
An analogous result for infinitesimal symmetries is presented in Theorem \ref{cKprojflat}:

\begin{thm} \label{thm-C}
Consider a modified PW metric associated to a projectively flat affine connection $D_A$ and a modification $\Ph_{AB}\in\ce_{(AB)}(2)$ such that $\B_2(\Ph)$ is generic. 
Locally, there is a bijective correspondence between its conformal Killing fields and triples $v^A \in \ce^A$, $\al_A\in\ce_A(2)$ and $\mathring{\ps} \in \mathbb{R}$ satisfying 
\begin{align}
& \bigl( D_{A} D_{B} v^C + \Rho_{AB} v^C \bigr)_0 =0 , \label{i-flat-v} \\
& \mathcal{L}_v (\B_2(\Ph)) = \mathring{\ps} \B_2(\Ph) , \label{i-flat-lie} \\
& D_{(A}\al_{B)}=0 . \label{i-flat-al}
\end{align}
The dimension of the space of conformal Killing fields equals to $d+\frac12 n(n+1)$, where $d$ is the dimension of 
the space of solutions to \eqref{i-flat-v} and \eqref{i-flat-lie} 
and $n$ is the dimension of the underlying projective manifold.
\end{thm}
\noindent
Here $\mathcal{L}_v$ denotes the Lie derivative in the direction of the vector field $v^A$ and 
the genericity of $\B_2(\Ph)$ means that this field, interpreted as a bundle map $\ce^A \to \ce_{[AB]C}$, is injective.
The summand $\frac12 n(n+1)$ in the previous count is the dimension of the solution space to \eqref{i-flat-al}, which is a (first) BGG equation corresponding to so-called projective Killing forms.

Another special cases treated in sections \ref{aEspecial} and \ref{cKspecial} concern modified conformal extensions of 2-dimensional projective structures.
In particular, we describe all possible dimensions of the spaces of almost Einstein scales and conformal infinitesimal symmetries in Theorems \ref{eAdim2} and \ref{cKdim2}, respectively.
In this context, an example with submaximal algebra of infinitesimal conformal symmetries is analyzed in section \ref{submax}.

In the last section \ref{Ambient}, we show that any modified PW metric admits a global Fefferman--Graham ambient metric.
In particular, the Fefferman--Graham obstruction tensor (which is the Bach tensor in dimension four) vanishes.
In this part, the modification tensor $\Phi$ enters quite innocently and the arguments are very similar to the non-modified situation studied in \cite{hsstz-ambient}.
The main statement, Theorem \ref{Theorem_ambient}, is reproduced as follows:

\begin{thm} 
Let $\ol{g}$ be a modified PW metric with the Schouten tensor $\ol\Rho$ and let $t$, $\rho$ be the additional coordinates on the ambient manifold.
Then 
\begin{align} 
\mb{g}=2\rho \d t\odot \d t + 2t \d t\odot \d\rho +t^2 \Big( \ol{g} +2\rho\,\ol\Rho \Big)
\end{align}
is a globally Ricci-flat Fefferman--Graham ambient metric of the conformal class of $\ol{g}$.
\end{thm}
\noindent
In addition, it is shown in Proposition \ref{Q-curvature} that the $Q$-curvature of any modified PW metric vanishes.

\subsubsection*{Acknowledgements}
The authors are grateful to Maciej Dunajski, Boris Kruglikov and Arman Taghavi-Chabert for valuable
discussion. JS was supported by the Czech science foundation (GA\v{C}R) under the grant GA19-06357S.
KS acknowledges funding received from the Norwegian Financial Mechanism 2014-2021, project registration number 
UMO-2019/34/H/ST1/00636.

\egroup

\section{Construction and characterization} 

In this section, we add details to the projective-to-conformal analogue of the modified PW construction sketched in the introduction.
After fixing notation and other conventions, we give auxiliary comparisons of various curvature quantities. 
This leads to a characterization of the flatness of the induced conformal structure and the structure itself, see Theorem \ref{relateFlat} and \ref{thm-char}, respectively.

\subsection{Preliminaries} \label{Prelim}

Most of the following preliminaries is taken from \cite{hsstz-walker}.
We start with a smooth manifold $M$ and pass to its (weighted) cotangent bundle, which we denote as $\wt{M}$. 
We use the standard abstract index notation so that the tensorial objects on $M$ and $\wt{M}$ are distinguished by the type of indices.
For instance, the tangent bundle $TM$ and $T\wt{M}$, as well as the space of its sections, is denoted as $\ce^A$ and $\wt\ce^a$, respectively. 
Symmetric and skew-symmetric tensors, as well as the corresponding operations, are denoted by round and square brackets, respectively, around indices.
For instance, $\al_{(AB)}\in\ce_{(AB)}$ denotes the symmetrization of a bilinear form $\al_{AB}\in\ce_{AB}$ on $M$, $\be^{ab}=\be^{[ab]}\in\wt\ce^{[ab]}$ denotes a bivector on $\wt{M}$ etc.

A projective structure on $M$ is given by an equivalence class of torsion-free affine connections, where two connections $D_A$ and $\wh{D}_A$ are projectively equivalent if there is a $1$-form $\Upsilon_A\in\ce_A$ such that
\begin{align}\label{eq-Dhat-D}
\widehat D_A\xi^B =D_A\xi^B+ \delta_{A}^B \Upsilon_{C}  \xi^C+\xi^C \Upsilon_{A} ,
\end{align}
for all $\xi^A\in\ce^A$.
Assuming the manifold $M$ is oriented and a volume form on $M$ is fixed, there is a unique connection in the projective class for which the volume form is parallel; any such connection is called \emph{special}.
Special connections are characterized by the fact that their Ricci (equivalently, Schouten) tensor is symmetric.
Adapted to the projective structure on $M$, we use appropriate parametrizations of density bundles as follows. 
For any $w\in\mbb{R}$, the density bundle of \emph{projective weight $w$} is defined as 
\begin{align}
\mc{E}(w):=\left(\wedge^n T M\right)^{-\frac{w}{n+1}} ,
\label{proj-w}
\end{align}
where $\dim M=n$. 
Weighted tensor bundles are denoted accordingly, e.g., $\ce_{A}(3)$ stays for $\ce_{A}\otimes\ce(3)$.

A conformal structure on $\wt{M}$ is given by an equivalence class of (pseudo-)Riemannian metrics, where two metrics $g_{ab}$ and $\wh{g}_{ab}$ are conformally equivalent if there is a function $f\in\wt\ce$ such that
\begin{align*}
\wh{g}_{ab} =e^{2f} g_{ab} .
\end{align*}
Similarly to \eqref{proj-w}, the density bundle of \emph{conformal weight $w$} on $\wt{M}$ is defined as
\begin{align}
\wt{\mc{E}}[w]:=\big(\wedge^{2n} T \wt{M} \big)^{-\frac{w}{2n}} ,
\label{conf-w}
\end{align}
where $\dim\wt{M}=2n$.
The conformal structure can be seen as a section of $\wt\ce_{(ab)}[2]$ so that the individual metrics from the 
conformal class correspond to  so-called \idx{scales}, i.e.\ everywhere positive sections of $\wt\ce[1]$.


The standard PW construction works as follows.
Let $D_A$ be a torsion-free affine connection on $M$ and let $\pi:T^*M\to M$ be the cotangent bundle.
Let $TT^*M=V\oplus H$ be the decomposition so that  $V\cong T^*M$ is the vertical distribution of the projection $\pi$ and $H\cong TM$ is the horizontal distribution determined by $D_A$.
The (standard) \emph{PW metric} associated to $D_A$ is the split-signature metric $\wt{g}_{ab}$ on $\wt{M}:= T^*M$ given by the natural pairing between $V$ and $H$ and requiring that both $V$ and $H$ are isotropic.
Since the distribution $V$ is parallel with respect to the corresponding Levi-Civita connection,  the PW metric is a Walker metric.

For a non-trivial symmetric bilinear form $\Phi_{AB}\in\ce_{(AB)}$ on $M$, let $\Phi_{ab}\in\wt\ce_{(ab)}$ be its pull-back to $\wt{M}$.
The PW metric can be modified so that 
\begin{align}
\ol{g}_{ab}:=\wt{g}_{ab}+\Phi_{ab}.
\label{eq-modPW}
\end{align}
The above mentioned properties change only so that the distribution $H$ is no more isotropic with respect to $\ol{g}_{ab}$.
The metric \eqref{eq-modPW} is called the \emph{modified PW metric} associated to the affine connection $D_A$ and the symmetric tensor $\Phi_{AB}$ on $M$.
Note that other names for  $\ol{g}_{ab}$ are used in the literature, e.g., it is called the Riemann extension metric in \cite{Patterson1952} and \cite{Derdzinski2009}, and the deformed Riemannian extension metric in \cite{Gilkey2009}.

Both the standard and the modified PW metric provide an identification of the tangent and cotangent bundle, $\wt\ce^a\cong\wt\ce_a$.
Per default, the indices are raised and lowered with respect to the standard PW metric $\wt{g}_{ab}$, the use of any other metric will always be mentioned explicitly.
In particular, the inverse metric of $\ol{g}_{ab}$ has the form $(\ol{g}^{-1})^{ab} = \wt{g}^{ab} - \Ph^{ab}$.

A coordinate expression of a modified PW metric is as follows.
For local coordinates $(x^A)$ on $M$, let $\Ga\ind{_A^C_B}$ be the Christoffel symbols of $D_A$ and let $(p_A)$ be the canonical fibre coordinates on $T^*M$.
Then the metric \eqref{eq-modPW} is expressed as
\begin{align}
\ol{g} =  2 \d x^A \odot \d p_A -2 \left(\Ga\ind{_A^C_B} \, p_C +\Phi_{AB}\right) \d x^A \odot \d x^B \, .
\label{formula-modPW}
\end{align}

There is a projective-to-conformal variant of the standard PW construction provided that we pass to an appropriately weighted cotangent bundle.
To do so, we assume that the manifold $M$ is oriented and the affine connection $D_A$ on $M$ is special, i.e. preserves a volume form.
This provides a trivialization of $\ce(w)$, i.e. an identification $T^*M(w)\cong T^*M$, for any $w$.
By \cite[Proposition 3.1]{hsstz-walker}, projectively equivalent torsion-free special connections on $M$ induce conformally equivalent PW metrics on $T^*M(w)$ if and only if $w=2$.
Thus, adjusting accordingly the previous setup, we have the following projectively invariant definition:

\begin{defn}
The \emph{modified conformal extension} or the \emph{modified conformal PW metric} associated to an oriented projective structure on $M$ and the symmetric weighted tensor $\Phi_{AB}\in\ce_{(AB)}(2)$ is the split-signature conformal structure on $\wt{M}:=T^*M(2)$ represented by the modified PW metric \eqref{eq-modPW} of a special torsion-free affine connection from the projective class.
\end{defn}

To investigate various quantities associated to the standard and the modified PW metric, we primarily need the relation between corresponding covariant derivatives.
By $\wt{D}_a$ and $\ol{D}_a$ we denote the Levi-Civita connection of the standard and the modified PW metric, respectively.
They are related as
\begin{align}
\ol{D}_a = \wt{D}_a + F_a\, \bullet , 
\label{wtol_F}
\end{align}
where the difference tensor $F_a{}^c{}_d \in \wt{\ce}_a{}^c{}_d$ is seen as a 1-form with values in the endomorphisms of $T\wt{M}$ and $\bullet$ denotes the algebraic action. 
This tensor can be specified as follows:

\begin{lem} \label{relateD}
Let the standard and the modified PW metric be related by \eqref{eq-modPW} and let the corresponding covariant derivatives be related by \eqref{wtol_F}.
Then
\begin{align}
F_a{}^c{}_d = \wt{D}_{(a} \Ph_{d)}{}^c - \tfrac12 \wt{D}^c \Ph_{ad} .
\label{_F}
\end{align}
In particular, $F_{acd}\in\wt\ce_{acd}$ is strictly horizontal, i.e. it is a section of $V\otimes V\otimes V$.
\end{lem}

\begin{proof}
The formula \eqref{_F} follows from the characterizing properties of the Levi-Civita connection, namely, from $\ol{D}_a\ol{g}_{bc}=0$ and the torsion-freeness, by a direct calculation.
The rest follows from \eqref{_F}, the strict horizontality of $\Phi_{ab}$ and the parallelity of $V$.
\end{proof}

For instance, for any $v^a\in\wt\ce^a$ and $\al_a\in\wt\ce_a$, we have
\begin{align}
\begin{split}
& \ol{D}_av^b = \wt{D}_av^b + v^r\wt{D}_{(a}\Phi_{r)}{}^b - \tfrac{1}{2}v^r \wt{D}^{b}\Phi_{ar} , \\ 
& \ol{D}_a\al_b = \wt{D}_a\al_b - \al_r\wt{D}_{(a}\Phi_{b)}{}^r + \tfrac{1}{2}\al_r \wt{D}^{r}\Phi_{ab} . \label{wtolD} 
\end{split}
\end{align}
Recall that the indices are raised with respect to the standard PW metric $\wt{g}_{ab}$.

\subsection{Curvature relations} \label{Curva}

There is a tight relation between the curvature tensors of the affine connection $D_A$ and the induced PW metric 
$\wt{g}_{ab}$. 
This, for example, dictates that $\wt{g}_{ab}$ is flat, Ricci-flat and conformally flat if and only 
if $D_A$ is flat, Ricci-flat and projectively flat, respectively.
These facts easily follow from explicit relations between the initial and the induced curvatures.
A particularly simple relation, which is often used below, is that the Ricci, respectively Schouten, tensor of $\wt{g}_{ab}$ is just the pull-back of the Ricci, respectively Schouten, tensor of $D_A$.
See \cite[Remark 3.5]{hsstz-walker} and the surrounding formulas for more details.

We are going to relate the respective tensors associated to the standard and the modified PW metric.
Some of these relations can also be found in the existing literature, cf. \cite[Section 8]{Patterson1952} or \cite[Theorem 1]{Afifi1954}. 
Nevertheless, we offer a self-contained derivation of these results and, notably, a conceptual interpretation of a condition characterizing the flatness of modified conformal extensions.
The objects associated to $\wt{g}_{ab}$ and $\ol{g}_{ab}$ are generally adorned by tilde and bar, respectively.

\begin{prop} \label{relateR}
Let the standard and the modified PW metric be related by \eqref{eq-modPW}.
Then the corresponding Riemann, Ricci and Schouten tensors are related by 
\begin{align}
\begin{split}
& \ol{R}_{ab}{}^r{}_d \bar{g}_{rc} = \wt{R}_{abcd}
- \wt{D}_c \wt{D}_{[a} \Ph_{b]d} +  \wt{D}_d \wt{D}_{[a} \Ph_{b]c} - \wt{R}_{ab}{}^r{}_{[c}\Ph_{d]r} , \\
& \ol{\Ric}_{ab} = \wt{\Ric}_{ab}, \qquad \ol{\Rho}_{ab} = \wt{\Rho}_{ab}.
\end{split}
\label{eq-RP}
\end{align}
In particular, the modified PW metric is Ricci-flat if and only if the original affine connection is Ricci-flat.
\end{prop}

\begin{proof}
With the relation \eqref{wtol_F}, it follows that 
$\ol{R}_{ab}{}^c{}_d = \wt{R}_{ab}{}^c{}_d + 2\wt{D}_{[a} F_{b]}{}^c{}_d$ since $F_{[a} \bullet F_{b]} \bullet =0$.
Substituting \eqref{_F} and contracting with $\ol{g}_{ab}$, the relation of Riemann tensors follows after some computation.
The Ricci and the Schouten relations follow easily, taking into account that $\Phi_{ab}$ is strictly horizontal.
The last statement follows from $\ol{\Ric}_{ab} = \wt{\Ric}_{ab}$ and the characterization of the Ricci-flatness of $\wt{g}_{ab}$.
\end{proof}

The basic curvature quantity in conformal geometry is the Weyl tensor, the conformally invariant part of the Riemann tensor.
For the metric $\ol{g}_{ab}$, it can be expressed as 
\begin{align}
\ol{W}_{ab}{}^r{}_d \bar{g}_{rc} &= \ol{R}_{ab}{}^r{}_d \ol{g}_{rc} 
- 4 \Proj_{[ab][cd]} \bigl( \ol{g}_{ac}  \ol{\Rho}_{bd} \bigr) ,
\label{eq-WRP}
\end{align}
where $\Proj_{[\cdot\,\cdot][\cdot\,\cdot]}$ denotes skew-symmetrization over the embraced indices.
Substituting \eqref{eq-modPW} and \eqref{eq-RP} allows to express the right-hand side of \eqref{eq-WRP} in terms of $\Phi_{ab}$ and tensors associated to $\wt{g}_{ab}$.
This can further be rearranged using the second BGG operator in the sequence
\begin{align*}
\wt{\ce}_{a}[2] \mathop{\longrightarrow}^{\wt\B_1} \wt{\ce}_{(ab)_0}[2] \mathop{\longrightarrow}^{\wt\B_2} \wt{\ce}_{[ab][cd]_0}[2] \longrightarrow \cdots  
\end{align*}
The first operator in this sequence is the conformal Killing operator, $\wt\B_1(v)_{ab} = \wt{D}_{(a} v_{b)_0}$, which is implicitly present in section \ref{Symm}.
The second operator is 
\begin{align} \label{wtB}
\begin{split}
\wt{\B}_2(\Phi)_{abcd} 
& = \Proj_{\boxplus_0} \bigl( \wt{D}_a \wt{D}_c\Phi_{bd}+\wt{\Rho}_{ac} \Phi_{bd}\bigr) = \\
& = \Proj_{[ab][cd]_0} \bigl( \wt{D}_a \wt{D}_c \Phi_{bd} + \wt{\Rho}_{ac} \Phi_{bd} + \tfrac14 \wt{W}_{ab}{}^r{}_c \Ph_{dr} - \tfrac14 \wt{W}_{cd}{}^r{}_a \Ph_{br} \bigr) ,
\end{split}
\end{align}
where $\Proj_{\boxplus_0}$ denotes the trace-free part of the `window' symmetry corresponding to the indicated Young tableau.
A short general introduction to the BGG theory and details for relevant projective BGG operators is given in appendix \ref{app-A}.
The role of the second BGG operator in measuring the change of the harmonic curvature (which is the Weyl curvature in conformal geometry) under deformations of parabolic geometries (which are represented by $\Phi_{ab}$ in our situation) is in general described in \cite[Theorem 3.6]{Cap2005b}.

\begin{prop} \label{relateW}
Let the standard and the modified PW metric be related by \eqref{eq-modPW}.
Then the corresponding Weyl tensors are related by 
\begin{align} \label{Weq}
\ol{W}_{ab}{}^r{}_d \ol{g}_{cr} = \wt{W}_{abcd} - 2\wt{\B}_2(\Phi)_{abcd}
- \tfrac{1}{2} \bigl( \wt{W}_{ab}{}^r{}_{[c} \Phi_{d]r}+\wt{W}_{cd}{}^r{}_{[a}\Phi_{b]r} \bigr) ,
\end{align}
where $\wt\B_2$ is the second BGG operator \eqref{wtB}.
In particular, the condition 
\begin{align} \label{intcon}
\ol{W}_{ab}{}^c{}_d  \, v^a w^d =0 
\end{align}
holds for all vertical vectors $v^a, w^d\in V$.
\end{prop}

\begin{proof}
Substituting \eqref{eq-modPW} and \eqref{eq-RP} into \eqref{eq-WRP} yields
\begin{align*}
\ol{W}_{ab}{}^r{}_d \bar{g}_{rc} &= \wt{W}_{abcd}
- 2 \Proj_{[ab][cd]} \bigl( \wt{D}_c \wt{D}_{[a} \Ph_{b]d} \bigr) - \wt{R}_{ab}{}^r{}_{[c}\Ph_{d]r}
-  4 \Proj_{[ab][cd]} \bigl( \Ph_{ac} \wt{\Rho}_{bd} \bigr) .
\end{align*}
The tensor $\Phi_{ab}$ is strictly horizontal, hence it is trace-free, $\Phi_{ab}\in\wt{\ce}_{(ab)_0}[2]$.
Using \eqref{wtB}, the formula \eqref{Weq} follows after some computation.
The condition \eqref{intcon} follows immediately from \eqref{Weq} and $\wt{W}_{abcd} \, v^a w^d =0$, which is one of the characteristic conditions of the standard PW metric, see \cite[Theorem~1]{hsstz-walker}.
\end{proof}

The conformal BGG operator \eqref{wtB} is visibly related to its projective counterpart in \eqref{BGGce_A2}.
Since  $\Ph_{ab}$ and $\wt{\Rho}_{ab}$ is the pullback of $\Ph_{AB}$ and $\Rho_{AB}$, respectively, also the tensor $\wt{\B}_2(\Phi)_{abcd}$ is the pullback of $\B_2(\Ph)_{ABCD}$.
Altogether, we conclude with

\begin{thm} \label{relateFlat}
Let the modified PW metric $\ol{g}_{ab}$ be induced by a special affine connection $D_A$ and a modification tensor $\Phi_{AB}$.
Then $\ol{g}_{ab}$ is conformally flat if and only if $D_A$ is projectively flat and $\B_2(\Phi)=0$, where $\B_2: \ce_{(AB)}(2) \to  \ce_{[AB][CD]}(2)$  is the second BGG operator \eqref{BGGce_A2}.
\end{thm}

\begin{proof}
For the standard PW metric, the relation between the projective and the conformal Weyl tensor, $W$ and $\wt{W}$, is described in \cite[equation (32)]{hsstz-walker}. 
In particular, $\wt{W}=0$ if and only if $W=0$.
From that relation and \eqref{Weq} we see that difference tensor $\ol{W}-\wt{W}$ is strictly horizontal. 
Hence, vanishing of $\ol{W}$ is equivalent to vanishing of both $\wt{W}$ and $\wt\B_2(\Phi)$.
Since $\wt\B_2(\Phi)$ is just the pullback of $\B_2(\Phi)$, the claim follows.
\end{proof}

\subsection{Characterization} \label{Char}

Our next aim is to provide a local characterization of split-signature conformal structures that arise as modified conformal extensions of projective structures.
We have already observed  that for such conformal structures the integrability condition \eqref{intcon} on the Weyl tensor holds.
Another ingredient needed for the characterization is a pure twistor spinor that emerges as follows.

The vertical distribution $V\subset T\wt{M}$ of the projection $\wt{M}\to M$ can always be represented by a pure spinor field $\chi$ such that $V=\ker\chi$.
In the case of standard conformal PW metrics, the spinor field can be chosen so that it satisfies the twistor spinor equation, which is a rather restrictive condition, see \cite[Theorem~1]{hsstz-walker}.

In fact, for a pure twistor spinor $\chi$ with an integrable kernel,  there are metrics in the conformal class for which $\chi$ is parallel and any two such metrics are related by a conformal factor constant along the leaves of  $\ker\chi$, see \cite[Proposition 4.2]{hsstz-walker}.
We will use this fact repeatedly in the article.

Investigating the influence of the modification gives the following 

\begin{lem} \label{lem-chi}
A modified conformal PW metric admits a pure twistor spinor $\chi$ such that $\ker\chi=V$.
\end{lem}

\begin{proof}
Let $\chi$ be the twistor spinor of the standard conformal PW metric such that $\ker\chi=V$.
As mentioned just before the Lemma, there is a scale for which $\chi$ is parallel.
By Lemma \ref{relateD}, it follows that $\chi$ is parallel also with respect to the modified PW metric.
In particular, it satisfies the corresponding twistor spinor equation.
\end{proof}

We are now ready to locally characterize modified conformal extensions of projective structures.
The proof is based on an analogous result for modified pseudo-Riemannian extensions of affine structures due to \cite{Afifi1954} and \cite{Derdzinski2009}.

\begin{thm} \label{thm-char} 
Locally, a conformal structure of split signature is a modified conformal extension of a projective structure if and only if the following properties are satisfied:
\begin{enumerate}[(a)]
\item there is a pure twistor spinor $\chi$  with integrable kernel $\ker\chi$,
\item the following integrability condition holds
\begin{align}\label{eq-weylcond}
{W}_{ab}{}^c{}_d \, v^a w^d & =0 , \quad \text{for all $v^a, w^d\in\ker\chi$} ,
\end{align}
where ${W}_{ab}{}^c{}_d$ is the conformal Weyl curvature.
\end{enumerate}
\end{thm}

\begin{proof}
By Proposition \ref{relateW} and Lemma \ref{lem-chi},  a conformal spin structure that arises as a modified conformal extension satisfies conditions (a) and (b). 

For the converse, let a split-signature conformal spin structure on $\wt{M}$ satisfy (a) and (b) and let $M$ be the local leaf space of the integrable distribution $\ker\chi \subset T\wt{M}$.
Here and below, all accounts are local and we implicitly assume that any neighborhood shrinks as needed.
By \cite[Proposition 4.2]{hsstz-walker}, there is a metric $g_{ab}$ in the conformal class such that $D_a\chi=0$, where $D_a$ is the Levi-Civita connection of $g_{ab}$.
Consequently, the Schouten tensor $\Rho_{ab}$ of $g_{ab}$ is strictly horizontal.
From the relation among the Weyl, Riemann and Schouten tensors, cf. \eqref{eq-WRP}, it follows that the integrability condition \eqref{eq-weylcond} is equivalent to 
\begin{align}\label{eq-riemcond}
R_{ab}{}^c{}_d \, v^a w^d & =0 , \quad \text{for all $v^a, w^d\in\ker\chi$} ,
\end{align}
where $R_{ab}{}^c{}_d$ is the Riemann curvature of $g_{ab}$.
The latter condition is the standard obstruction for the connection $D_a$ to descend to an affine connection $D_A$ on the local leaf space $M$.
It follows from \cite[Theorem 4.5]{Derdzinski2009} that a section of the projection $\wt{M}\to M$ (interpreted as a zero section) provides a local diffeomorphism $\wt{M}\cong T^*M$ under which the metric $g_{ab}$ corresponds to a modified PW metric of the affine connection $D_A$.

From \cite[Proposition 4.2]{hsstz-walker} we also know that any two metrics leaving $\chi$ parallel are related by a conformal factor constant along the leaves of $\ker\chi$.
The transformation of the Levi-Civita connections under such rescaling then shows that the descended affine connections on $M$ are projectively related, cf. the proof of \cite[Proposition 4.5]{hsstz-walker}.
Any descended connection is necessarily special, since the Ricci tensor is symmetric.
The corresponding volume form provides a trivialization of any density bundle so, in particular, an identification $T^*M\cong T^*M(2)$.
Altogether, the conformal spin structure on $\wt{M}$ satisfying (a) and (b) determines a projective structure on the local leaf space $M$ and, under the local identification $\wt{M}\cong T^*M(2)$,  its modified conformal extension corresponds to the initial conformal structure.
\end{proof}

\begin{rema} \label{rem-char}
Comparing the current characterization with the one for standard conformal PW metrics in \cite[Theorem~1]{hsstz-walker}, we see that the latter are distinguished by the presence of a conformal Killing field $k\in\ker\chi$ satisfying 
\begin{align}\label{kchi}
\mc{L}_k\chi=-\tfrac{1}{2}(n+1)\chi ,
\end{align}
where $n$ is the dimension of the underlying projective manifold.
After a modification, the vector field $k$ still satisfies \eqref{kchi} but need not be a conformal Killing field; 
the conformal Killing operator is related to the modification term $\Phi$ as follows:
\begin{align} \label{S2f}
\ol{D}_{(a}k_{b)_0} =-\Phi_{ab} ,
\end{align}
where $\ol{D}_a$ is the Levi-Civita connection of the modified PW metric $\ol{g}_{ab}=\wt{g}_{ab}+\Phi_{ab}$ and the trace-free part is taken with respect to $\ol{g}_{ab}$.
These facts are related to the property \eqref{k-homothety} recalled below and the transformation formula \eqref{wtolD}.
\end{rema}

The field $k\in\ker\chi$ can be changed so that the conditions \eqref{kchi} and \eqref{S2f} still hold.
Indeed, putting 
\begin{align}
k'_a := k_a+\al_a, \quad
\Phi'_{AB} := \Phi_{AB}^{} - D_{(A}\al_{B)},
\label{eq-k-Phi}
\end{align}
where $\al_a \in \wt{\mc{E}}_a[2]$ is the pull-back of $\al_A\in\mc{E}_A(2)$ with respect to the projection $\wt{M}\to M$, the Lie derivative is $\mc{L}_{k'}\chi=-\tfrac{1}{2}(n+1)\chi$, since $\wt{D}_a \al_b$ is strictly horizontal.
Further, one easily verifies that  $\ol{D}_{(a}^{}k'_{b)_0} =-\Phi'_{ab}$, where $\Phi'_{ab} \in \wt{\mc{E}}_{(ab)}[2]$ is the pull-back of $\Phi'_{AB}\in\ce_{(AB)}(2)$.
The freedom in the choice of $\al_A\in\mc{E}_A(2)$ corresponds to the freedom in the choice of local sections of the projection $\wt{M}\to M$ as in the proof of Theorem \ref{thm-char}.
%

Altogether, two modified conformal PW metrics of the same projective structure whose modification tensors are related by \eqref{eq-k-Phi} are basically comparable.
Note that the difference $\Phi_{AB}-\Phi'_{AB}=D_{(A}\al_{B)}$ is the image of the first BGG operator \eqref{BGGce_A2} on $\al_A\in\mc{E}_A(2)$.
In particular, the current observations matches nicely with Theorem \ref{relateFlat}.

\subsection{Further conventions}

In the rest of this section, we collect necessary spinor-related notions that are repeatedly used below.
This is a digest of \cite[section~2]{hsstz-walker}, which is based on a general calculus developed in \cite{Taghavi-Chabert2012}.

Let $D_A$ be a special torsion-free affine connection on $M$, 
let $\wt{g}_{ab}$ be the induced PW metric on $\wt{M}$ and let $\wt{\mc{S}}_+$ and $\wt{\mc{S}}_-$ be the irreducible spinor bundles.
To distinguish, we decorate the respective sections with primed and unprimed capital indices.
In particular, the related gamma matrices are denoted as $\gamma\ind{_a^{B'}_A} \in \wt\ce_a\otimes \wt{\mc{S}}_+ \otimes \wt{\mc{S}}_-^*$ and $\gamma\ind{_a^B_{A'}}\in \wt\ce_a\otimes \wt{\mc{S}}_- \otimes \wt{\mc{S}}_+^*$.
The above mentioned pure spinor field annihilating $V\subset T\wt{M}$ is written as $\chi^{A'}\in\wt{\mc{S}}_+$.
Similarly, there is a pure spinor field $\check\eta_{A'}\in\wt{\mc{S}}_+^*$ annihilating $H\subset T\wt{M}$, the horizontal distribution given by $D_A$, such that $\chi^{A'} \check{\eta}_{A'}\ne 0$.
To the spinors $\chi^{A'}$ and $\check\eta_{A'}$, we associate the sections 
\begin{align*}
\chi_a^A := \gamma\ind{_a^A_{B'}} \, \chi^{B'} \in \wt\ce_a\otimes\wt{\mc{S}}_- , \quad
\check{\eta}_{aA} := {\check{\eta}}_{B'} \, \gamma\ind{_a^{B'}_A} \in  \wt\ce_a\otimes\wt{\mc{S}}_-^* .
\end{align*}
Besides $V \cong \ker\chi_a^A$ and $H \cong \ker\check\eta_{aA}$, we can identify $V \cong \im\check\eta_{aA}$ and $H \cong \im\chi_a^A$.
The freedom in the choice of $\chi^{A'}$ and $\check\eta_{A'}$ can be fixed by 
\begin{align}\label{eq-for-Josef}
\chi^{aA} \wt{D}_a = \parderv{}{p_A} \, , \quad \check{\eta}^a_A \wt{D}_a = \parderv{}{x^A} + \Gamma \ind{_A^C_B} \, p_C \, \parderv{}{p_B} \, ,
\end{align}
where $\wt{D}_a$ is the Levi-Civita connection of $\wt{g}_{ab}$.

The previous setup allows the following useful interpretations:
the field $\chi^A_a$, respectively $\check\eta_A^a$, is seen as the pull-back of 1-forms, respectively the horizontal lift of vector fields, from $M$ to $\wt M$.
In particular, the PW metric can be written as
\begin{align}
\wt g_{ab} = 2\chi_{(a}^A \check{\eta}\ind{_{b)A}}
\label{g-chi-eta}
\end{align}
and its defining properties are reflected in 
\begin{align}\label{xinu}
\chi_a^A \chi^{aB} = 0 \, , \quad  \check{\eta}^a_A \check{\eta}_{aB} = 0 \, , \quad \chi_a^A \check{\eta}^a_B = \delta_A^B \, .
\end{align}
For later use, note also that the pull-back of the modification term $\Phi_{AB}\in\ce_{(AB)}(2)$ is 
\begin{align}
\Phi_{ab} = \Phi_{AB} \chi_a^A \chi_b^B 
\label{phi-chi-chi}
\end{align}
and the horizontal--vertical decomposition of a vector field $v^a\in\wt\ce^a$ looks like
\begin{align} \label{hor-vert}
v^a = \ups^A \check{\eta}^a_A + \be_A\chi^{aA} ,
\end{align}
where the coefficientts are interpreted as $\ups^A\in\ce^A$ and $\be_A\in\ce_A(2)$.
For vector fields \eqref{hor-vert} with coefficients $\ups^A$ and $\be_A$ depending only on $x^A$, the covariant derivative has the form
\begin{align}\label{Dv-no-p}
 \wt{D}_a {v}^b & = \left( D_A \ups^B \right) \chi_a^A \check{\eta}^b_B
 + \left( D_A \be_B  - \ups^C R \ind{_{CB}^D_A} p_D \right) \chi_a^A \chi^{bB} \, .
\end{align}

Concerning the distinguished vertical vector field from Remark \ref{rem-char}, its standard coordinate expression is $k = 2 p_A \parderv{}{p_A}$.
We will also need the following relations
\begin{align}
k^a = 2 p_A \chi^{aA} , \quad
p_A=\tfrac12k^r\check\eta_{rA} ,
\label{p-k-eta}
\end{align}
where, compared with \eqref{hor-vert}, $p_A$ is interpreted as a section of $\ce_A(2)$.
In fact, $k^a$ is a conformal Killing field of satisfying 
\begin{align}
\wt{D}_a k_b = \mu_{ab} +\wt{g}_{ab} .
\label{k-homothety}
\end{align}
where $\mu_{ab}=\wt{D}_{[a} k_{b]}$.
In particular, $k^a$ is a homothety of $\tilde{g}_{ab}$.
Finally, we note that the endomorphism $\mu^a{}_b$ acts as plus and minus the identity on the horizintal and the vertical distribution, respectively, i.e., 
$$
\mu^a{}_r \chi^r{}^A = -\chi^a{}^A, \quad 
\mu^a{}_r \check{\eta}^r{}^A = \check\eta^a{}^A.
$$

\section{Einstein metrics} \label{Einstein}

The aim of this section is to characterize Einstein representative metrics of modified conformal extensions in terms of underlying projective data. 
This section culminates in Theorem \ref{summaryaE}.
Contrary to the non-modified situation, Einstein metrics in the conformal class are not necessarily Ricci-flat.

An \emph{almost Einstein scale} of a conformal structure is a scale $\si \in \wt{\ce}[1]$ such that the corresponding metric in the conformal class is Einstein off the zero set of $\si$. 
For conformal class represented by a metric $\ol{g}_{ab}$, 
the scale $\si$ is almost Einstein if and only if the trace-free part of $(\ol{D}_{a} \ol{D}_{b} + \ol\Rho_{ab})\si$ with respect to $\ol{g}_{ab}$ vanishes, where $\ol{D}_a$ and $\ol{\Rho}_{ab}$ are the Levi-Civita connection and the Schouten tensor of $\ol{g}_{ab}$, respectively. 
If we consider a modified conformal extension represented by $\ol{g}_{ab}=\wt{g}_{ab}+\Phi_{ab}$, this condition can be also written in terms of the non-modified PW metric $\wt{g}_{ab}$ as 
\begin{equation} \label{aEscales}
\bigl( \wt{D}_{a} \wt{D}_{b} + \wt\Rho_{ab} \bigr) \si = \psi (\wt{g}_{ab} +\Phi_{ab}) ,
\end{equation}
for some $\psi \in \wt{\ce}[-1]$ that need not be specified for now.

In order to characterize solutions of \eqref{aEscales} in underlying projective terms, we employ the first projective BGG operators from \eqref{BGGce1} and \eqref{BGGce^A(-1)}.
Solutions of the former BGG equation are the scales $\ta\in\ce(1)$ that determine Ricci-flat affine connections from the projective class, solutions of the latter BGG equation are the weighted vector fields $\xi^A\in\ce^A(-1)$ satisfying
\begin{align}
D_A\xi^B -\tfrac1n\de_A{}^B D_R\xi^R =0 .
\label{xi-1}
\end{align}
As a consequent condition we get
\begin{align}
W_{AB}{}^C{}_R \xi^R =0 ,
\label{xi-weyl}
\end{align}
cf.\ \eqref{compBGGce^A(-1)}.
Note that the prolongation of \eqref{xi-1} gives
\begin{align}
D_{(A} D_{B)} \xi^R + \de^R_{(A}\Rho_{B)S} \xi^S =0 ,
\label{xi-prolong}
\end{align}
cf.\ \cite[eqn.~(54)]{hsstz-walker}.
Moreover, we shall need the bilinear differential operator $\mc{F}: \ce^A(-1) \times  \ce_{(AB)}(2) \to \ce_{(AB)}(1)$ given by 
\begin{equation}
\mc{F}(\xi,\Ph)_{AB} = \xi^R \bigl( D_{(A} \Ph_{B)R} - \tfrac12 D_R\Ph_{AB} \bigr) + \tfrac{1}{n} (D_R\xi^R) \Ph_{AB} ,
\label{Bop}
\end{equation}
which, indeed,  is projectively invariant too.

\begin{prop} \label{downaE}
A section $\si \in \wt{\ce}[1]$ is  an almost Einstein scale of the modified conformal extension represented by $\ol{g}_{ab}=\wt{g}_{ab}+\Phi_{ab}$ if and only if $\si$ is linear in $p_A$, 
\begin{equation} \label{polaE}
\si= \xi^Rp_R + \ta
\end{equation}
and the underlying sections $\ta \in \ce(1)$ and $\xi^A \in \ce^A(-1)$ satisfy the integrability  condition
\begin{align} 
\xi^R W_{RB}{}^C{}_D=0 \label{taxi0} 
\end{align}
and the differential conditions
\begin{align}
& \left(D_A \xi^B\right)_0 =0,  \label{taxi1} \\
& \left(D_{A} D_{B} + \Rho_{AB}\right) \ta = \mc{F}(\xi, \Ph)_{AB}. \label{taxi2}
\end{align}
\end{prop}

Note that, for $\Phi_{AB}=0$, the right-hand side of \eqref{taxi2} vanishes, i.e., both $\ta$ and $\xi^A$ are  solutions of the corresponding BGG equations.

\begin{proof}
Let $\si \in \wt{\ce}[1]$ be an almost Einstein scale, i.e., the equation \eqref{aEscales} be satisfied.
Both $\ol{g}_{ab}=\wt{g}_{ab}+\Phi_{ab}$ and $\wt\Rho_{ab}$ vanish when contracting with two vertical vectors, hence the previous assumption implies $\chi^{aA} \chi^{bB}  \wt{D}_{a} \wt{D}_{b} \si = \frac{\del^2}{\del p_A\del p_B} \si =0$.
Thus, $\si$ is a linear polynomial in $p_A$, for which we fix the notation as in \eqref{polaE}.

To express the Einstein scale equation in terms of $\ta \in \ce(1)$ and $\xi^A \in \ce^A(-1)$,
one has to substitute  \eqref{polaE} into \eqref{aEscales} and expand according to \eqref{Dv-no-p}, taking into account \eqref{p-k-eta}.
Considering the decomposition of the left-hand side as
\begin{equation*} \label{decmaE}
\bigl( \wt{D}_{a} \wt{D}_{b} + \wt\Rho_{ab} \bigr) \si 
= \Th'_A{}^B \chi_{(a}{}^A\check{\eta}_{b)B} + \Th''_{AB} \chi_{(a}{}^A \chi_{b)}{}^B ,
\end{equation*}
direct computation reveals that
\begin{align*}
& \Th'_A{}^B = 2 D_A \xi^B, \\
& \Th''_{AB} = \big( D_{(A} D_{B)} \xi^R + \de^R_{(A}\Rho_{B)S} \xi^S - \xi^S W_{S(A}{}^R{}_{B)} \big) \, p_R \, + \\
&  \hskip5em +  (D_{A}D_{B} + \Rho_{AB})\ta - \xi^R ( D_{(A}\Ph_{B)R} - \tfrac12 D_R\Ph_{AB} ) .
\end{align*}
Recasting the right-hand side of \eqref{aEscales} according to \eqref{g-chi-eta} and \eqref{phi-chi-chi}, the Einstein scale equation is written as
\begin{align}
\big( \Th'_A{}^B -2\psi\de_A{}^B \big) \chi_{(a}{}^A\check{\eta}_{b)B} + \big( \Th''_{AB} -\psi\Phi_{AB} \big) \chi_{(a}{}^A \chi_{b)}{}^B = 0 ,
\label{aE-chi-eta}
\end{align}
where 
$\psi = \tfrac1{n} D_R\xi^R$.
The two summands in \eqref{aE-chi-eta} must vanish separately.
Taking into account the fact that all expressions are polynomial in $p_A$, 
the first condition $\Th'_A{}^B -2\psi\de_A{}^B =0$ is equivalent to 
\eqref{taxi1}
and the second condition $\Th''_{AB} -\psi\Phi_{AB} =0$ is equivalent to the pair
\begin{align}
& D_{(A} D_{B)} \xi^R + \de^R_{(A}\Rho_{B)S} \xi^S - \xi^S W_{S(A}{}^R{}_{B)} = 0, \label{condaE2} \\
& (D_{A}D_{B} + \Rho_{AB})\ta - \xi^R ( D_{(A}\Ph_{B)R} - \tfrac12 D_R\Ph_{AB} ) - \tfrac1n (D_R\xi^R) \Phi_{AB} =0. \label{condaE3}
\end{align}
From \eqref{condaE2} and the condition \eqref{xi-prolong}, which is a consequence of \eqref{taxi1}, we conclude that $\xi^S W_{S(A}{}^R{}_{B)} = 0$.
From the symmetries of Weyl tensor and the condition \eqref{xi-weyl}, which is another consequence of \eqref{taxi1}, we conclude that $\xi^S W_{S[A}{}^R{}_{B]} =-\frac12 W_{AB}{}^R{}_S \xi^S =0$.
Thus, the integrability condition \eqref{taxi0} holds.
Conversely, \eqref{taxi0} and \eqref{taxi1} clearly imply \eqref{condaE2}.
With the notation from \eqref{Bop}, the equation \eqref{condaE3} is just \eqref{taxi2}.
\end{proof}

Proposition \ref{downaE} establishes a bijective correspondence between almost Einstein scales and pairs $\ta \in \ce(1)$ and $\xi^A \in \ce^A(-1)$ satisfying certain projectively invariant conditions. 
The following theorem shows that the components can also be identified in conformal terms, and we have full control over the Ricci-flatness of rescaled metrics.

\begin{thm} \label{summaryaE}
There is a bijective correspondence between almost Einstein scales of the modified conformal extension and pairs $\ta \in \ce(1)$ and $\xi^A \in \ce^A(-1)$ satisfying conditions \eqref{taxi0}--\eqref{taxi2}.
More precisely, an almost Einstein scale $\si \in \wt{\ce}[1]$ can be uniquely decomposed as
\begin{align*}
\si = \si_+ + \si_- ,
\end{align*}
where $\mc{L}_k\si_+ = \si_+$ and $\mc{L}_k\si_- = -\si_-$.
These components correspond to \eqref{polaE} as 
\begin{align}
\xi^A = \chi^{aA}\wt{D}_a\si_+ , \quad
\ta = \si_- .
\label{tauxi}
\end{align}

Moreover, the scalar curvature of the rescaled metric corresponding to $\si$ is 
\begin{align*}
2n(2n-1)\Ph_{RS}\, \xi^R\xi^S , 
\end{align*}
off its zero set, where $n$ is the dimension of the underlying projective manifold and $\Phi_{AB} \in \ce_{AB}(2)$ is the modification tensor.
\end{thm}

\def\mgi{(\ol{g}^{-1})}

\begin{proof}
For the decomposition of an almost Einstein scale as in \eqref{polaE}, let $\si_+ := \xi^R p_R$ and $\si_- := \ta$.
The properties $\mc{L}_k\si_+ = \si_+$ and $\mc{L}_k\si_- = -\si_-$ follow by the same argument as in \cite[Lemma~5.2]{hsstz-walker}
(the degree of homogeneity with respect to $p_A$ plays a key role there).
The expressions in \eqref{tauxi} are clear, recalling that $\chi^{aA}\wt{D}_a = \frac{\del}{\del p_A}$.

The remaining part is based on a similar reasoning as in \cite[Proposition 5.3]{hsstz-walker} with regard to our current setting as in the proof of Proposition \ref{downaE}.
The scalar curvature of a metric is proportional to the trace of its Schouten tensor with the constant factor $2(2n-1)$.
This follows from definitions which, together with transformation formulas reflecting a change of scale, can be found e.g. in \cite{Bailey1994}.
Let $\ol{g}_{ab}=\wt{g}_{ab}+\Ph_{ab}$ be the modified PW metric and $\ol{\Rho}:=\mgi^{rs}\ol\Rho_{rs}$ be the trace of the corresponding Schouten tensor.
Analogous quantities corresponding to an Einstein scale $\si$ are denoted as $\wh{g}_{ab}$ and $\wh\Rho$, respectively.
The key relation is
\begin{align}
\wh\Rho = \ol\Rho -\mgi^{rs} (\ol{D}_r\Ups_s +(n-1) \Ups_r\Ups_s) ,
\label{eq-sc}
\end{align}
where $\Upsilon_a = -\si^{-1} \ol{D}_a \si$.
Since $\ol\Rho_{ab}$ is strictly horizontal, $\ol\Rho$ vanishes.
Similarly, for the decomposition $\si = \si_+ + \si_-$ as above, the component $\si_-$ does not contribute to the trace.
Hence we may consider just $\si =\si_+ =\xi^R p_R$.
To compute the differential, we use \eqref{p-k-eta}, expand according to \eqref{Dv-no-p} and substitute \eqref{xi-1}, which yields
\begin{align*}
\ol{D}_a \si_+ = \tfrac{1}{2n} (D_R \xi^R) k_a + \xi^R \check{\eta}_{aR} .
\end{align*}
The trace of the derivative of $\Ups_a$  simplifies as 
\begin{align*}
\mgi^{rs} \ol{D}_r \Upsilon_s = \mgi^{rs} \Upsilon_r \Upsilon_s - 2\si^{-1}_+ D_R \xi^R .
\end{align*}
Putting things together, one verifies that the divergence terms vanish and \eqref{eq-sc} reduces to 
$\wh\Rho = n\, \Ph_{RS}\,\xi^R\xi^S$.
Hence the claim follows.
\end{proof}

We conclude this section with an additional relation that will be helpful in the next section where we discuss some special cases.
It is a consequence of the conditions from Proposition \ref{downaE}:

\begin{cor} \label{E2BGG}
If $\ta\in\ce(1)$, $\xi^A\in\ce^A(-1)$ and $\Ph_{AB}\in\ce_{(AB)}(2)$ satisfy \eqref{taxi0}--\eqref{taxi2} then
\begin{align}
\left( W_{AB}{}^R{}_C D_R - Y_{CAB} \right) \ta = 
\xi^R \left( \tfrac34 W_{AB}{}^S{}_C \Ph_{RS} -2 \B_2(\Ph)_{ABCR} \right) , 
\label{taxi3}
\end{align}
where $\B_2: \ce_{(AB)}(2) \to  \ce_{[AB][CD]}(2)$ is the second BGG operator \eqref{BGGce_A2}.
\end{cor}

\begin{proof}
Applying the second BGG operator \eqref{BGGce1} to the left-hand side of \eqref{taxi2} yields
$$
- \tfrac12 \bigl( W_{AB}{}^R{}_C D_R - Y_{CAB} \bigr) \ta ,
$$
where $Y_{CAB} = 2 D_{[A} \Rho_{B]C}$ is the Cotton tensor, cf.\ \eqref{compBGGce1}.
Applying the same operator to the right-hand side of \eqref{taxi2}, a tedious computation leads to
$$
\xi^R \bigl( \B_2(\Ph)_{ABCR} - \tfrac38 W_{AB}{}^S{}_C \Ph_{RS} 
+ \tfrac14 W_{CR}{}^S{}_{[A} \Ph_{B]S} \bigr) ,
$$
where $\B_2$ is the BGG operator \eqref{BGGce_A2}, more explicitly described in \eqref{B2expl}.
Here one has to take into account that $\xi^A$ satisfies \eqref{taxi1} and, consequently,  \eqref{xi-weyl}.
The previous two displays together with the integrability condition \eqref{taxi0} give the stated result.
\end{proof}

\section{Einstein metrics: special cases} \label{aEspecial}

The conditions from Proposition \ref{downaE} significantly simplify in certain special cases.
We discuss the case of modified conformal extensions of projectively flat structures, 
the case when the modification term $\Phi_{AB}$ is in the image of the first BGG operator and the lowest dimensional case.
In all these cases, we are able to untangle the interrelations between the source sections $\ta$ and $\xi^A$, relative to the modification term $\Phi_{AB}$, and specify the dimension of the space of almost Einstein scales.

\subsection{Projectively flat case} \label{Einstein-flat}
In this case, many curvature related objects disappear which  leads to significant simplifications.
Primarily, both the Weyl and the Cotton tensor vanish.
In particular, the condition \eqref{taxi0} is satisfied trivially and there is only one term in \eqref{taxi3} that survives.
Also, in the flat case any BGG sequence is a complex and, working locally, it is actually exact.
These facts lead to the following splitting of the characterizing conditions:

\begin{thm} \label{aEprojflat}
Consider a modified PW metric associated to a projectively flat affine connection $D_A$ and a modification tensor $\Ph_{AB}\in\ce_{(AB)}(2)$. 
Locally, there is a bijective correspondence between almost Einstein scales of the conformal class and pairs $\ta \in \ce(1)$ and $\xi^A \in \ce^A(-1)$ satisfying 
\begin{align}
& (D_{A} D_{B} + \Rho_{AB}) \ta = 0 , \label{flat-ta} \\
& (D_A \xi^B)_0=0 , \quad \xi^R\, \B_2(\Ph)_{ABCR}=0 , \label{flat-xi}
\end{align}
where $\B_2: \ce_{(AB)}(2) \to  \ce_{[AB][CD]}(2)$ is the second BGG operator \eqref{BGGce_A2}.
The dimension of the space of almost Einstein scales equals to $d+n+1$, where $d$ is the dimension of the space of solutions to \eqref{flat-xi} and $n$ is the dimension of the underlying projective manifold.
\end{thm}

\begin{proof}
Only the conditions \eqref{taxi1} and \eqref{taxi2} from Proposition \ref{downaE} are relevant in our case.
Applying the second BGG operator to \eqref{taxi2} yields $0=\xi^R\, \B_2(\Ph)_{ABCR}$, cf.\ Corollary \ref{E2BGG}.

Locally, by the exactness of the BGG sequence, the right-hand side of \eqref{taxi2} is in the image of the first BGG operator on a section of $\ce(1)$. 
This guarantees the existence of $\ta \in \ce(1)$ satisfying \eqref{taxi2} and all such sections are parametrized by solutions to the equation \eqref{flat-ta}.
However, in the flat case, solutions to the later equation allow a coordinate expression $\ta=c_A x^A + c_0$, where $c_0, c_1, \dots, c_n$ are arbitrary constants, cf. section \ref{A1}.
\end{proof}

If the tensor field $\B_2(\Ph)$ is generic, the dimension of the space of almost Einstein scales equals to $n+1$.
Here the genericity means that the field, interpreted as a bundle map $\ce^A \to \ce_{[AB]C}$, is injective.

\subsection{Special modification} \label{Einstein-spec}
Here we assume that $\Ph_{AB} = D_{(A}\ph_{B)}$, for some $\ph_A \in \ce_A(2)$, i.e.\ $\Ph_{AB}$ is in the image of the first BGG operator \eqref{BGGce_A2}. 
With regard to Remark \ref{rem-char}, the characterization of almost Einstein scales has to correspond to the one for standard conformal extensions as in \cite[Theorem 2]{hsstz-walker}.
To keep the presentation self-contained, we derive the characterization directly from Proposition \ref{downaE}.
The effect of $\Ph_{AB}$ is so that the corresponding system of equations, which is homogeneous in the standard case, becomes non-homogeneous and one seeks for a particular solution.
In particular, the space of almost Einstein scales of a modified conformal extension of the current type is an affine space over the vector space of almost Einstein scales of its non-modified companion.

\begin{thm} \label{aEimage}
Let the standard PW metric be modified by the term of the form  $\Ph_{AB} = D_{(A} \ph_{B)}$, for some $\ph_A \in \ce_A(2)$.
There is a bijective correspondence between almost Einstein scales of the conformal class and pairs $\ta \in \ce(1)$ and $\xi^A \in \ce^A(-1)$ satisfying 
\begin{align}
& (D_{A} D_{B} + \Rho_{AB}) \ta = 0 , \label{spec-ta} \\
& \xi^R W_{RA}{}^C{}_B=0 , \quad (D_A \xi^B)_0=0 . \label{spec-xi}
\end{align}
The dimension of the space of almost Einstein scales equals to $d_1+d_2$, where $d_1$ and $d_2$ is the dimension of the space of solutions to \eqref{spec-ta} and \eqref{spec-xi}, respectively.
\end{thm}

\begin{proof}
A straightforward computation reveals that \eqref{Bop} has the form
\begin{align}
\mc{F}(\xi,\Ph)_{AB} = \tfrac12 (D_AD_B + \Rho_{AB}) (\xi^R \ph_R) .
\end{align}
Thus, $\ta=\frac12 \xi^R \ph_R$ is a particular solution to \eqref{taxi2} and all such solutions are parametrized by solutions to the equation \eqref{spec-ta}.
The equations  \eqref{spec-xi} are just the remaining conditions from Proposition \ref{downaE}.
\end{proof}

Note that, compared with the standard extension, 
the modification tensor enters the game so that the rescaled metrics need not be Ricci-flat, see Theorem \ref{summaryaE}.

\subsection{Dimension four} \label{Einstein-2}
Here we examine 4-dimensional modified conformal extensions, i.e. those of 2-dimensional projective structures.
In this case, the Weyl tensor $W_{AB}{}^C{}_D$ vanishes automatically and the key curvature invariant is the Cotton tensor $Y_{CAB}$.
This simplifies the condition \eqref{taxi3}, namely, 
\begin{align}
Y_{CAB}\, \ta = 2\xi^R\, \B_2(\Ph)_{ABCR} .
\label{dim2-taxi}
\end{align}
Also, $\B_2(\Ph)$ is a section of the bundle $\ce_{[AB][CD]}(2)$ which is---in this dimension---a density bundle.
Using the projective volume form $\bep_{AB} \in \ce_{[AB]}(3)$, respectively its inverse $\bep^{AB} \in \ce^{[AB]}(-3)$, it is identified with $\ce(-4)$.
In particular, equation \eqref{dim2-taxi} can be rewritten as 
\begin{align} 
(\star Y)^{A}\, \ta = 2 \xi^A (\star \B_2(\Ph)) ,
\label{dim2-taxi-star}
\end{align}
where $(\star Y)^A := Y_{CDE} \bep^{AC}\bep^{DE} \in \ce^A(-6)$ and $\star \B_2(\Ph) := \B_2(\Ph)_{ABCD} \bep^{AB}\bep^{CD} \in \ce(-4)$. 

\begin{thm} \label{eAdim2}
For modified conformal extensions of 2-dimensional projective structures, 
the dimension of the space of almost Einstein scales is 0, 1, 3 or 6.
\end{thm}

\begin{proof}
Let the dimension of the space of almost Einstein scales be denoted by $d$.
Whenever we consider sections that do not vanish identically, we restrict off their zero sets.

For non-flat projective structures, the solution space to \eqref{taxi1} is at most 1-dimensional;
this follows from \eqref{compBGGce_A2} and the neighbouring discussion.
Indeed, we have $\ce^A(-1) \cong \ce_A(2)$ thus the equation \eqref{taxi1} corresponds to the first BGG equation \eqref{BGGce_A2}, for which \eqref{compBGGce_A2} gives the related integrability condition.
Any solution $\xi^A\in\ce^A(-1)$ determines $\ta\in\ce(1)$ uniquely via the condition \eqref{dim2-taxi-star}. 
Such a pair defines an almost Einstein scale if and only if it satisfies \eqref{taxi2}.
Altogether, we conclude that $d\le 1$ in such cases.

For flat projective structures and conformal extensions with $\B_2(\Ph)\ne 0$, it follows from \eqref{dim2-taxi-star} that $\xi^A=0$.
The conditions from Proposition \ref{downaE} are reduced to the single equation $(D_{A} D_{B} + \Rho_{AB}) \ta = 0$, whose solution space is 3-dimensional, cf. the proof of Theorem \ref{aEprojflat}.
Thus, $d=3$ in such cases.

For flat projective structures and conformal extensions with $\B_2(\Ph)=0$, the induced conformal structure is flat, cf.\ Theorem \ref{relateFlat}.
In such cases, the dimension $d$ is maximal possible, which is well known to be $d=6$.
\end{proof}

All values listed in Theorem \ref{eAdim2} are realizable:
\begin{enumerate}[(1)]
\item[(0)] 
For conformal extensions of generic projective structures there are no almost Einstein scales.
\item[(1)] 
For the projective structure given by the Levi-Civita connection of a generic surface of revolution, let us consider its standard (non-modified) conformal extension. 
The Killing field generating the rotation gives rise to the vector field satisfying \eqref{taxi1}, see \cite[section 7.1]{hsstz-walker}, and this is the only source for almost Einstein scales.
\item[(3)] 
A generic modified conformal extension of flat projective structure has 3-dimensional space of almost Einstein scales; a particular example of this type is discussed in section \ref{submax}.
\item[(6)]
The standard conformal extension of the flat projective structure has 6-dimensional space of almost Einstein scales.
\end{enumerate}

\section{Symmetries} \label{Symm}

In this section, we characterize infinitesimal conformal symmetries of a modified PW metric in terms of underlying projective data. Before we come to the main statement in Theorem \ref{summaryCKE}, we need several preparatory observations.


As the first step, we express the conformal Killing equation and its prolongation for the modified PW metric in terms of the original (non-modified) one.
For a conformal Killing field  ${v}^a$ of the modified PW metric $\ol{g}_{ab}=\wt{g}_{ab}+\Phi_{ab}$, the trace-free part of $\ol{D}_{(a}\ol{v}_{b)}$ with respect to $\ol{g}_{ab}$ vanishes, where $\ol{D}_a$ is the Levi-Civita connection of $\ol{g}_{ab}$ and $\ol{v}_b = v_b + v^r\Phi_{rb}$. 
Compactly written,
\begin{align}
\ol{D}_{(a}\ol{v}_{b)} -\psi\ol{g}_{ab} =0 ,
\label{CKE}
\end{align}
where $\psi$ is the divergence that need not be specified at the moment.
An expansion of this condition in the sense of Lemma~\ref{relateD} gives the needed expressions.
In comparison with the non-modified case, an extra term ($\om_{ab}$) and its further derivatives ($\om'_{abc}$ and $\om''_{ab}$) are contained in the final formulas:

\begin{lem}\label{relateCKE}
Let $v^a \in \wt{\ce}^a$ be a conformal Killing field of the modified PW metric $\ol{g}_{ab}=\wt{g}_{ab}+\Phi_{ab}$
and let us decompose
\begin{equation} \label{CKEol}
\wt{D}_a v_b = {\mu}_{ab} + {\ps} \wt{g}_{ab} + \om_{ab},
\end{equation}
where ${\mu}_{ab} = \wt{D}_{[a} v_{b]}$, ${\ps} = \tfrac{1}{2n} \wt{D}_r v^r$ 
and $\om_{ab} = \wt{D}_{(a} v_{b)_0}$.
Then:
\begin{enumerate}[(i)]
\item The symmetric trace-free part of $\wt{D}_a v_b$ is
\begin{align}
\om_{ab} = -\tfrac12(\mathcal{L}_v \Ph)_{ab} .
\label{CKEom}
\end{align}
In particular, it is strictly horizontal, i.e.,  $\chi^{aA}\om_{ab}=\chi^{bB}\om_{ab}=0$.
\item Differential consequences of \eqref{CKEol} are
\begin{align}
\wt{D}_a {\ps} &= \wt{\Rho}_{ar} v^r - {\be}_a , \label{CKEprol1} \\
\wt{D}_{a} {\mu}_{bc} &= -2\wt{g}_{a[b} {\be}_{c]} - 2\wt{\Rho}_{a[b}v_{c]} - \wt{W}_{bcar}v^r + \om'_{abc} , \label{CKEprol2} \\  
\wt{D}_a {\be}_b &= - \wt{Y}_{abr}v^r - {\ps} \wt{\Rho}_{ab} + \wt{\Rho}_a{}^r {\mu}_{rb} + \om''_{ab} ,\label{CKEprol3} 
\end{align}
for some ${\be}_a\in\wt\ce_a$, $\om'_{abc}\in\wt\ce_{a[bc]}[2]$ and $\om''_{ab}\in\wt\ce_{ab}$ satisfying
$\chi^{aA} \om'_{abc}=0$ and $\chi^{bB} \om''_{ab}=0$, respectively.
\end{enumerate}
\end{lem}

\begin{proof}
(i) 
On the one hand,  it follows from \eqref{wtolD} that, for $\ol{v}_b = v_b + v^r\Phi_{rb}$,
\begin{align*}
\ol{D}_a \ol{v}_b
&= \wt{D}_a v_b - v_r  \wt{D}_{(a} \Ph_{b)}{}^r + \tfrac12 v^r \wt{D}_r \Ph_{ab} + (\wt{D}_{a} v^r) \Ph_{br} + v^r \wt{D}_a \Ph_{br} = \\
&= \wt{D}_a v_b + \tfrac12 v^r \wt{D}_r \Ph_{ab} + (\wt{D}_{a} v^r) \Ph_{br} + v_r \wt{D}_{[a} \Ph_{b]}{}^r .
\end{align*}
The conformal Killing equation \eqref{CKE} then reads as
\begin{equation} \label{CKE-ol}
\wt{D}_{(a} v_{b)} + \tfrac12 v^r \wt{D}_r \Ph_{ab} + (\wt{D}_{(a} v^r) \Ph_{b)r} - \psi(\wt g_{ab}+\Phi_{ab}) =0,
\end{equation}
where ${\ps} = \tfrac{1}{2n} \wt{D}_r v^r$.
On the other hand, the Lie derivative of $\Phi_{ab}\in\wt\ce_{ab}[2]$ in the direction of $v^a$ is
\begin{align}
(\mathcal{L}_v \Ph)_{ab} = v^r \wt{D}_r \Ph_{ab}  + 2(\wt{D}_{(a} v^r) \Ph_{b)r} - \tfrac{1}{n} (\wt{D}_rv^r) \Ph_{ab} 
\label{Lder}
\end{align}
(where the last coefficient reflects the conventions from \eqref{conf-w}).
Thus, the equation \eqref{CKE-ol} can be written as
\begin{align*}
\wt{D}_{(a} v_{b)} + \tfrac12(\mathcal{L}_v \Ph)_{ab} - \psi\wt{g}_{ab} = 0 , 
\end{align*}
which gives \eqref{CKEom}.
Since $\Phi_{ab}$ is strictly horizontal and the flow of $v^a$ preserves the vertical distribution, the rest follows.

(ii) 
To derive the differential consequences of \eqref{CKEol}, we shall mimic the computation of standard prolonged systems, cf.\ \cite{Gover2008}. 
Equation \eqref{CKEprol1} is read as the definition of $\be_a$.
Applying $\wt{D}_c$ to \eqref{CKEol}, commuting covariant derivatives on the left-hand side and skewing over $b$ and $c$, the equation \eqref{CKEprol2} follows after some manipulation.
From the computation it further follows that $\om'_{abc} = 2\wt{D}_{[b} \om_{c]a}$.
Since $\om_{ab}$ is strictly horizontal, the property $\chi^{aA} \om'_{abc}=0$ holds.

Further, applying $\wt g^{ab}$ to \eqref{CKEprol2} yields
\begin{align}
\wt{D}^r \wt{\mu}_{ra} = -(2n-1) \wt{\be}_a + \wt{\Rho}_{ar}v^r + \wt{D}^r \om_{ra}.
\label{CKEpom}
\end{align}
Finally, applying $\wt{D}^c$ to \eqref{CKEprol2}, commuting covariant derivatives and using \eqref{CKEpom}, the equation \eqref{CKEprol3} follows after a tedious but straightforward computation.
From the computation it further follows that $\om''_{ab}$ is a linear combination of $\wt D_r\wt D^r \om_{ab}$ and $\wt{D}_a \wt{D}^r \om_{br}$.
In particular, the property $\chi^{bB} \om''_{ab}=0$ holds.
\end{proof}


As the next step, to  express infinitesimal conformal symmetries in underlying projective terms, we shall need several projectively invariant operators.
Primarily, we employ the first BGG operators from \eqref{BGGce^AB(-2)} and  \eqref{BGGce_A2} and the operator \eqref{BGGwt}.
The corresponding equations, respectively their solutions, are expounded as follows:
The equation associated to \eqref{BGGce^AB(-2)} is equivalent to the system 
\begin{align}
& D_A w^{BC} - 2 \de_{A}{}^{[B} \nu^{C]} =0 , \quad \text{where} \quad \nu^C = \tfrac{1}{n-1} D_Rw^{RC}, 
\label{wopProl1} \\
& D_A \nu^B + \Rho_{AR}w^{RB} + \tfrac{1}{2(n-2)}  w^{RS} W_{RS}{}^B{}_A=0 .
\label{wopProl2}
\end{align}
Solutions associated to \eqref{BGGwt} and \eqref{BGGce_A2} are the infinitesimal projective symmetries and the so-called projective \textit{Killing forms}, respectively.
Moreover, we shall need the bilinear differential operator
$\mc{F}:\ce^{[CD]}(-2) \times \ce_{(AB)}(2) \to \ce^C{}_{(AB)}$ given by
\begin{align} 
\label{B'op}
\mc{F}(w, \Ph )^C{}_{AB} = w^{RC} \bigl( D_{(A} \Ph_{B)R} - \tfrac12 D_R \Ph_{AB} \bigr) + \nu^C\Ph_{AB} , 
\end{align}
where $\nu^C$ is as above.
One can check it is projectively invariant too.
Analogously to \eqref{Lder}, we also have the Lie derivative on sections of $\ce_{(AB)}(2)$, 
\begin{equation} 
\label{Lop}
(\mathcal{L}_v \Ph)_{AB} = 
v^R D_R \Ph_{AB} + 2 (D_{(A} v^R) \Ph_{B)R} - \tfrac{2}{n+1} (D_Rv^R) \Ph_{AB} ,
\end{equation}
where $v^R \in \ce^R$
(the last coefficient reflects the conventions from \eqref{proj-w}).

We are ready to characterize infinitesimal conformal symmetries of a modified PW metric. 
The description is based on the horizontal--vertical decomposition of conformal vector fields whose components are identified with projectively invariant objects:

\begin{prop} \label{downCKE}
Let $v^a = \ups^A \check{\eta}^a_A + \be_A\chi^{aA}$ 
be the decomposition of a vector field with respect to $\wt{g}_{ab}$ as in \eqref{hor-vert}.
Then $v^a$ is a conformal Killing field of the modified PW metric $\ol{g}_{ab}=\wt{g}_{ab}+\Phi_{ab}$ if and only if $\ups^A$ and $\be_A$ are polynomials in $p_A$,
\begin{equation} \label{polCKE}
\ups^A = w^{AB} p_B + v^A, \quad \be_A = \ps_A{}^{BC} p_Bp_C + \ph_A{}^B p_B + \al_A,
\end{equation}
and the underlying sections
$w^{AB} \in \ce^{AB}(-2)$, $v^A\in\ce^A$, $\ps_A{}^{BC} \in \ce_A{}^{(BC)}(-2)$, $\ph_A{}^B \in \ce_A{}^B$, $\al_A \in \ce_A^{}(2)$
satisfy the algebraic condition
\begin{align}
w^{(AB)} =0 , 
\label{wpsph-1}
\end{align}
the integrability condition
\begin{align} 
w^{R(C} W_{R(A}{}^{D)}{}_{B)} =0, \label{wpsph0}
\end{align}
the differential conditions
\begin{align} 
& \left(D_A w^{BC}\right)_0 =0, \label{wpsph1} \\
& \left( D_A D_B v^C + \Rho_{AB} v^C + v^R W_{R(A}{}^C{}_{B)} \right)_0 = 
- \left( \mc{F}(w,\Ph)^C{}_{AB} \right)_0 , \label{wpsph2} \\
& D_{(A}^{} \al_{B)} = -\tfrac12 (\mathcal{L}_v \Ph)_{AB} +\tfrac12 \mathring{\ps} \, \Ph_{AB} , \label{wpsph3} 
\end{align}
for some $\mathring{\ps} \in\ce$, such that 
\begin{align}
D_A \mathring{\ps} = \tfrac2{n+1}\,\mc{F}(w,\Ph)^R{}_{AR} , 
\label{wpsph4}
\end{align}
where $\nu^C = \tfrac{1}{n-1} D_R w^{RC}$ as in \eqref{wopProl1}, 
and the remaining coefficients are given by
\begin{align}
& \ps_A{}^{BC} =\de_A{}^{(B} \nu^{C)}, \label{wpsph5} \\
& \ph_A{}^B =  - \left( D_A v^B  +w^{SB} \Ph_{SA} \right)_0 
+  \left( \tfrac{n-1}{n(n+1)} (D_S v^S) + \mathring{\ps} \right)  \de_A{}^B .
\label{wpsph6} 
\end{align}
\end{prop}

Note that, for $\Phi_{AB}=0$, the right-hand sides of \eqref{wpsph2}--\eqref{wpsph4} vanish, i.e., 
each of $w^{AB}$, $v^A$,  $\al_A$ is a solution to the corresponding BGG equation and $\mathring{\psi}$ is constant.
Further comments relating the previous description to the non-modified case are at the end of section.

\begin{proof}
Let $v^a = \ups^A \check{\eta}^a_A + \be_A\chi^{aA}$ be a vector field of the form \eqref{polCKE}.
We start by expressing the conformal Killing equation of $\ol{g}_{ab}=\wt{g}_{ab}+\Phi_{ab}$ in terms of the underlying objects 
$w^{AB}$, $\ps_A{}^{BC}$, $\ph_A{}^B$ and $\al_A$.
The proper weights declared in the statement follow from the discussion around \eqref{p-k-eta} and \eqref{hor-vert} and need not be emphasized below.

From the proof of Lemma \ref{relateCKE} we know that \eqref{CKE} is equivalent to 
\begin{align*}
\ol{D}_{(a}\ol{v}_{b)} = \wt{D}_{(a} v_{b)} + \tfrac12(\mathcal{L}_v \Ph)_{ab} + \psi\Phi_{ab} ,
\end{align*}
where $\ol{v}_b = v_b + v^r \Ph_{rb}$ and ${\ps} = \tfrac{1}{2n} \wt{D}_r v^r$.
Now, one has to substitute \eqref{polCKE}, respectively \eqref{phi-chi-chi}, and expand according to \eqref{Dv-no-p}, taking into account \eqref{p-k-eta}.
Considering the decomposition of the form
\begin{equation*} \label{decmCKE}
\ol{D}_{(a} \bar{v}_{b)} = \Th^{AB}\, \check{\eta}_{(a|A|} \check{\eta}_{b)B} + \Th'_A{}^B \, \chi_{(a}{}^A\check{\eta}_{b)B} + \Th''_{AB}\, \chi_{(a}{}^A \chi_{b)}{}^B ,
\end{equation*}
direct, albeit rather lengthy, computation reveals that
\begin{align}
\Th^{AB} &= w^{(AB)}, \label{Th1} \\
\Th'_A{}^B &= \big( D_A w^{BR} + 2\ps_A{}^{BR} \big) \, p_R + \big( D_A v^B  +w^{SB} \Ph_{SA} + \ph_A{}^B \big), \label{Th2}  \\
\begin{split}
\Th''_{AB} &= \big( D_{(A} \ps_{B)}{}^{RS} - w^{T(R} W_{T(A}{}^{S)}{}_{B)} +w^{T(R} \Rho_{T(A}\de^{S)}{}_{B)} \big) \, p_Rp_S + \, \\
& \hskip4em + \big( D_{(A} \ph_{B)}{}^R + D_{(A} (w^{SR}\Ph_{B)S})  - v^S W_{S(A}{}^R{}_{B)} -v^R \Rho_{AB} + \\
& \hskip8em + v^S \Rho_{S(A} \de^R{}_{B)}  - w^{SR} D_{(A}\Ph_{B)S} + \tfrac12 w^{SR} D_{S}\Ph_{AB} \big) \, p_R \, + \\
& \hskip4em +\big( D_{(A} \al_{B)} + (D_{(A} v^R) \Ph_{B)R} + \tfrac12  v^S D_{S}\Ph_{AB} \big) .  \label{Th3}
\end{split}
\end{align}
With the reference to \eqref{g-chi-eta} and \eqref{phi-chi-chi}, the conformal Killing equation \eqref{CKE}
reads as
\begin{align}
\Th^{AB}\, \check{\eta}_{(a|A|} \check{\eta}_{b)B} + \big( \Th'_A{}^B -2\psi\de_A^B \big) \chi_{(a}{}^A\check{\eta}_{b)B} + \big( \Th''_{AB} -\psi\Phi_{AB} \big) \chi_{(a}{}^A \chi_{b)}{}^B = 0 ,
\label{CKE-chi-eta}
\end{align}
where 
\begin{align}
\psi =\tfrac1{2n} \big( D_S w^{SR} + 2\ps_S{}^{SR} \big) p_R + \tfrac1{2n} \big( D_S v^S + \ph_S{}^S \big).
\label{psi}
\end{align}
The three summands in \eqref{CKE-chi-eta} must vanish separately.
Taking into account the fact that the expressions \eqref{Th1}--\eqref{Th3} and \eqref{psi} are polynomial in $p_A$, 
the first condition $\Th^{AB} =0$ is
\begin{align}
w^{(AB)}=0 , \label{condCKE0}
\end{align}
the second condition $\Th'_A{}^B -2\psi\de_A^B =0$ is equivalent to the pair 
\begin{align}
& D_A w^{BR} + 2\ps_A{}^{BR} - \tfrac{1}{n} \de_A^B \bigl( D_S w^{SR} + 2\ps_S{}^{SR} \bigr) =0,  \label{condCKE1} \\
& \bigl( D_A v^B  +w^{SB} \Ph_{SA} + \ph_A{}^B \bigr)_0 =0 , \label{condCKE2}
\end{align}
and the third condition $\Th''_{AB} -\psi\Phi_{AB}=0$ is equivalent to the triple
\begin{align} 
&  D_{(A} \ps_{B)}{}^{RS} - w^{T(R} W_{T(A}{}^{S)}{}_{B)}  +w^{T(R} \Rho_{T(A}\de^{S)}{}_{B)} =0, \label{condCKE3} \\
\begin{split}
& D_{(A}^{} \ph_{B)}{}^R + D_{(A} (w^{SR}\Ph_{B)S})  - v^S W_{S(A}{}^R{}_{B)} -v^R \Rho_{AB} + 
\\ & \hskip2em 
+ v^S \Rho_{S(A} \de^R{}_{B)} - w^{SR} D_{(A}\Ph_{B)S} + \tfrac12 w^{SR} D_{S}\Ph_{AB} -\tfrac1{2n} \big( D_S w^{SR} + 2\ps_S{}^{SR} \big) = 0, \label{condCKE4} 
\end{split}
\\
& D_{(A} \al_{B)} + (D_{(A} v^R) \Ph_{B)R} + \tfrac12  v^S D_{S}\Ph_{AB} - \tfrac1{2n} \big( D_S v^S + \ph_S{}^S \big)\Phi_{AB} =0. \label{condCKE5} 
\end{align}
The system \eqref{condCKE0}--\eqref{condCKE5} provides the desired characterization of conformal Killing fields of the form \eqref{polCKE} in purely underlying terms.
Now we analyze individual conditions in detail:

The condition \eqref{condCKE0} is just \eqref{wpsph-1}, i.e., $w^{AB}$ is a bivector.

Skew-symmetrization of \eqref{condCKE1} over $B$ and $R$ gives $(D_A w^{BR})_0=0$, which is just the condition~\eqref{wpsph1}.

Symmetrization of \eqref{condCKE1} over $B$ and $R$ gives $(\ps_A{}^{BR})_0=0$, which means that $\ps_A{}^{BC} =\de_A{}^{(B} \nu^{C)}$ for some $\nu^C$.
Using this notation and taking the trace of \eqref{condCKE1} give
\begin{align}
D_R w^{RA} = (n-1)\nu^A ,
\label{nu}
\end{align}
i.e., $\nu^C$ is as stated and we get the condition \eqref{wpsph5}.

Equation \eqref{condCKE2} dictates the trace-free part of $\ph_A{}^B$, while its trace is undetermined. 
It will be convenient to express it as 
\begin{align}
\ph_R{}^R = \tfrac{n-1}{n+1} D_R v^R +  n\mathring{\ps} ,
\label{ringpsi}
\end{align}
for some function $\mathring{\ps}\in\ce$. 
Thus, we have the equation \eqref{wpsph6}.

With \eqref{wpsph5}, the equation \eqref{condCKE3} is rewritten as 
\begin{align} 
\de_{(A}{}^{(R} D_{B)}\nu^{S)} - w^{T(R} W_{T(A}{}^{S)}{}_{B)}  +w^{T(R} \Rho_{T(A}\de^{S)}{}_{B)} =0. \label{condCKE3'}
\end{align}
In particular, the trace-free part of $w^{T(R} W_{T(A}{}^{S)}{}_{B)}$ vanishes.
Taking the full trace of \eqref{condCKE3'} and comparing with \eqref{wopProl2}, which is a consequence of \eqref{wpsph1}, gives  
\begin{align}
w^{TR} W_{TR}{}^B{}_A=0.
\label{traceW}
\end{align}
Thus, $w^{T(R} W_{T(A}{}^{S)}{}_{B)}=0$, which is just the integrability condition \eqref{wpsph0}.

With \eqref{wpsph5} and \eqref{nu}, the equation \eqref{condCKE4} is rewritten as 
\begin{align} 
\begin{split}
& D_{(A}^{} \ph_{B)}{}^R + D_{(A} (w^{SR}\Ph_{B)S})  - v^S W_{S(A}{}^R{}_{B)} -v^R \Rho_{AB} + \\
& \hskip6em + v^S \Rho_{S(A} \de^R{}_{B)} - w^{SR} D_{(A}\Ph_{B)S} + \tfrac12 w^{SR} D_{S}\Ph_{AB} -\nu^R \Ph_{AB} = 0, \label{condCKE4'} 
\end{split}
\end{align}
The trace-free part of \eqref{condCKE4'} is built from operators \eqref{BGGwt} and \eqref{B'op} as stated in \eqref{wpsph2}, whereas the trace of \eqref{condCKE4'} leads to \eqref{wpsph4}.
For the later claim, one has to use the following consequence of \eqref{wpsph6}, respectively \eqref{ringpsi},
\begin{equation*} \label{Dph'}
D_{(A} \ph_{R)}{}^R = \tfrac12 (n+1) D_A \mathring{\ps} - \tfrac12 (n-1) \Rho_{AR} v^R -\tfrac12 D_R ( w^{SR} \Ph_{SR}). 
\end{equation*}

With \eqref{nu} and \eqref{ringpsi}, the equation \eqref{condCKE5} is rewritten as 
\begin{align} 
& D_{(A} \al_{B)} + (D_{(A} v^R) \Ph_{B)R} + \tfrac12  v^S D_{S}\Ph_{AB} - \tfrac{1}{n+1} (D_S v^S) \Ph_{AB} -\tfrac12 \mathring\ps \Ph_{AB} =0. \label{condCKE5'} 
\end{align}
Referring to \eqref{Lop}, this is just the equation \eqref{wpsph3}.

Altogether, we have derived all the conditions \eqref{wpsph-1}--\eqref{wpsph6} in the statement from \eqref{condCKE0}--\eqref{condCKE5}.
Retracing, respectively adapting, the previous account, it is easy to proceed in the opposite direction.
Hence the two systems are equivalent.

To finish the proof, we have to show that any conformal Killing field of the modified PW metric has the form \eqref{polCKE}.
We shall mimic part of the computation from the proof of \cite[Proposition 6.5]{hsstz-walker}, taking into account the current modifications gathered in Lemma \ref{relateCKE}.
Let the scale be chosen so that the twistor spinor $\chi$ is parallel with respect to $\wt{D}$.
Differentiating \eqref{CKEol} and substituting \eqref{CKEprol1} and \eqref{CKEprol2} give
\begin{align*}
\wt{D}_a\wt{D}_b v_c = 
- 2\wt{g}_{a[b} {\be}_{c]} - 2\wt{\Rho}_{a[b}v_{c]} - \wt{W}_{bcar}v^r + \om'_{abc} 
+ \wt{\Rho}_{ar} v^r\wt{g}_{bc} - {\be}_a\wt{g}_{bc}
+ \wt{D}_a\om_{bc} .
\end{align*}
After vertical contractions we get
\begin{align}
\chi^{aA}\chi^{bB}\wt{D}_a\wt{D}_b v_c =
2\chi^{aA}\chi^{bB}\wt{g}_{c[a}\be_{b]}.
\label{chichiDDv}
\end{align}
In particular, another contraction gives
$\chi^{aA}\chi^{bB}\chi^{cC} \wt{D}_a\wt{D}_b v_c = \frac{\del^2}{\del p_A\del p_B} \ups^C = 0$, 
i.e., $\ups^C$ is a linear polynomial in $p_A$.
Further differentiation of \eqref{chichiDDv}, substitution of \eqref{CKEprol3} and the vertical contraction yield
\begin{align*}
\chi^{aA}\chi^{bB}\chi^{cC}\wt{D}_a\wt{D}_b\wt{D}_c v_d =0.
\end{align*}
In particular, we have
$\chi^{aA}\chi^{bB}\chi^{cC}\check{\eta}_D^d \wt{D}_a\wt{D}_b\wt{D}_c v_d = \frac{\del^3}{\del p_A\del p_B\del p_C} \be_D = 0$, 
i.e., $\be_D$ is a polynomial of second order in $p_A$.
\end{proof}

The vector field $v^a = \ups^A \check{\eta}^a_A + \be_A\chi^{aA}$ of the form \eqref{polCKE} can alternatively be written as $v^a = v^a_+ + v^a_0 + v^a_-$, where
\begin{align}
& v_+^a =  w^{AB} p_B \check{\eta}^a_A + \ps_A{}^{BC} p_B p_C \chi^{aA} \, , \label{lift+} \\
& v_0^a =v^A \check{\eta}_A^a  + \ph_A{}^B p_B \chi^{aA} \, , \label{lift0}  \\
& v_-^a = \al_A \chi^{aA} \, . \label{lift-} 
\end{align}
The key coefficients are $w^{AB}$, $v^A$, $\al_A$ and $\mathring{\ps}$, whereas the remaining two, $\ps_A{}^{BC}$ and $\ph_A{}^B$, are determined by the others via \eqref{wpsph5} and \eqref{wpsph6}.
Though the two prescriptions are not projectively invariant, formulas \eqref{lift0} and \eqref{lift+} provide well-defined lifts of underlying objects, i.e., the vector fields are indeed independent of the choice of affine connection in the projective class (the invariance of \eqref{lift-} is clear).
For the non-modified case, the independence is shown in \cite[Lemma 6.1]{hsstz-walker}, respectively \cite[Remark 6.10]{hsstz-walker}.
The modification $\Phi_{AB}$ enters only the component $v^a_0$, according to \eqref{wpsph6}, which clearly does not spoil the projective invariance.

We are now prepared to state the main theorem of this section.
\begin{thm} \label{summaryCKE}
There is a bijective correspondence between conformal Killing fields of the modified PW metric and quadruples 
\begin{align}
w^{AB} \in \ce^{[AB]}(-2), \quad v^A \in \ce^A, \quad \al_A \in \ce_{A}(2), \quad \mathring{\ps} \in\ce 
\label{blocks}
\end{align}
satisfying the conditions \eqref{wpsph0}--\eqref{wpsph4}.
Locally, the last condition is equivalent to 
\begin{align} \label{wpsph7}
w^{RS} \bigl( D_{[A} D_{|R} \Ph_{S|B]} + \Rho_{[A|R} \Ph_{S|B]} \bigr)=0.
\end{align}

More precisely, a conformal Killing field $v^a \in \wt\ce^a$ can be uniquely decomposed as
\begin{align}
v^a = v^a_+ + v^a_0 + v^a_- , 
\label{CKE+0-}
\end{align}
where $\mc{L}_k v^a_{+} = 2 v^a_{+}$, $\mc{L}_k v^a_{-} = -2 v^a_{-}$ and $\mc{L}_k v^a_0 = 0$.
The components can be expressed as the lifts \eqref{lift+}--\eqref{lift-}, where the essential coefficients are 
\begin{align}
\begin{split}
& w^{AB} = \chi^{aA} \chi^{B}_b \wt{D}_{a} v_{+}^b , \quad
v^A = \chi^{A}_b v_{0}^b , \\
& \al_A = \check{\eta}_{aA}  v_{-}^a , \quad
\mathring\ps = \tfrac{1}{n+1} \left( \tfrac1n \wt{D}_b v^b_{0} - \mu^a{}_b \wt{D}_a v^b_{0} \right) , 
\end{split}
\label{bits}
\end{align}
the remaining ones being given by \eqref{wpsph5} and \eqref{wpsph6}.
\end{thm}

\begin{proof}
The first part of the statement is clear from Proposition \ref{downCKE}.

The condition \eqref{wpsph4} means that the 1-form 
\begin{align}
\mc{F}(w,\Ph)^R{}_{AR} = \nu^R \Ph_{AR} - w^{RS} D_R \Ph_{SA}
\label{GRR}
\end{align}
is exact. 
Locally, it is equivalent to its closedness.
Putting the differential of \eqref{GRR} equal to zero and applying \eqref{wopProl1}, \eqref{wopProl2}, \eqref{traceW} lead to \eqref{wpsph7}.

Given the decomposition \eqref{CKE+0-} along \eqref{lift+}--\eqref{lift-} of a vector field $v^a = \ups^A \check{\eta}^a_A + \be_A\chi^{aA}$ of the form \eqref{polCKE}, 
the properties $\mc{L}_k v^a_{+} = 2 v^a_{+}$, $\mc{L}_k v^a_{-} = -2 v^a_{-}$ and $\mc{L}_k v^a_0 = 0$ follow from  \cite[Lemma~6.4]{hsstz-walker}.
(Note that just the degree of homogeneity with respect to $p_A$ plays a role here.)

The expressions of $v^A$, $\al_A$ and $w^{AB}$ in \eqref{bits} are clear, cf.\ the omnipresent relations \eqref{xinu} and $\chi^{aA}\wt{D}_a = \frac{\del}{\del p_A}$.
The expression for  $\mathring\ps$ can be verified by a direct computation using the form of $v_0^a$ from \eqref{lift0}. 
Specifically, one computes that
$$
\wt{D}_b v_0^b = \bigl( 1 + \tfrac{n-1}{n+1} \bigr) D_Rv^R + n \mathring\ps 
\quad \text{and} \quad \mu^a{}_b \wt{D}_a v^b_{0} = D_Rv^R - \ph_R{}^R .
$$
Hence the formula for $\mathring\ps$ follows using \eqref{wpsph6}.
\end{proof}

For $\Phi_{AB}=0$, the function $\mathring\ps$ has to be constant and we can single out the distinguished vector field $k^a$ from \eqref{lift0}.
This---as well as each of the other components---is a conformal Killing field and we recover \cite[Theorem~3]{hsstz-walker}, respectively  \cite[Proposition 6.5]{hsstz-walker}.
In particular, each component is given solely by one of the sections from \eqref{blocks}.

For general modifications, none of the components of the decomposition above  has to be a conformal Killing field.
Also, considering the sources \eqref{blocks} individually, the corresponding lifts may, but need not, be conformal Killing fields.
On the one hand, for a projective Killing form $\al_A$, respectively an infinitesimal projective symmetry $v^A$ satisfying $\mc{L}_v\Phi =0$, the lift has the form $v^a=v^a_{-}$, respectively $v^a=v^a_{0}$, and is an infinitesimal conformal symmetry, 
cf.\ conditions \eqref{wpsph2} and \eqref{wpsph3}.
On the other hand, a bivector $w^{AB}$ such that the right-hand  side of \eqref{wpsph2}, respectively \eqref{wpsph4}, does not vanish cannot lift to an infinitesimal symmetry without a counterbalancing influence of $v^A$, respectively $\mathring\ps$.

\section{Symmetries: special cases} \label{cKspecial}

The conditions from Proposition \ref{downCKE} significantly simplify in some special cases.
As in  section \ref{aEspecial}, we focus on the case of modified conformal extensions of projectively flat structures, the case when the modification term $\Phi_{AB}$ is in the image of the first BGG operator and the lowest dimensional case, respectively.
In all these cases, we attempt to specify the dimension of the algebra of infinitesimal symmetries.
Contrary to section \ref{aEspecial}, the details are getting more technical, which is why we combine several specifications in section \ref{Symm2}.
In any case, we are still able to obtain interesting partial results.
In this context, we also discuss an example with submaximal algebra of conformal Killing fields in section \ref{submax}.

\subsection{Projectively flat case} \label{sym-flat}
As in section \ref{Einstein-flat}, we explore some consequences of  projective flatness.
The simplifications are obtained by applying second BGG operators to particular conditions in Proposition \ref{downCKE} and local exactness of BGG complexes.
However, we have to manage more involved conditions than in the case of almost Einstein scales, which leads us to impose an additional genericity assumption.

\begin{thm} \label{cKprojflat}
Consider a modified PW metric associated to a projectively flat affine connection $D_A$ and a modification $\Ph_{AB}$ such that $\B_2(\Ph)$ is generic, where $\B_2: \ce_{(AB)}(2) \to  \ce_{[AB][CD]}(2)$  is the second BGG operator \eqref{BGGce_A2}.
Locally, there is a bijective correspondence between its conformal Killing fields and triples $v^A \in \ce^A$, $\al_A\in\ce_A(2)$ and $\mathring{\ps} \in \mathbb{R}$ satisfying 
\begin{align}
& \bigl( D_{A} D_{B} v^C + \Rho_{AB} v^C \bigr)_0 =0 , \label{flat-v} \\
& \mathcal{L}_v (\B_2(\Ph)) = \mathring{\ps} \B_2(\Ph) , \label{flat-lie} \\
& D_{(A}\al_{B)}=0 . \label{flat-al}
\end{align}
The dimension of the space of conformal Killing fields equals to $d+\frac12 n(n+1)$, where $d$ is the dimension of the space of solutions to \eqref{flat-v} and \eqref{flat-lie} and $n$ is the dimension of the underlying projective manifold.
\end{thm}

The genericity of $\B_2(\Ph)$ means that this field, interpreted as a bundle map $\ce^A \to \ce_{[AB]C}$, is injective.
Recall the condition \eqref{flat-v} indicates that $v^A$ is an infinitesimal projective symmetry.

\begin{proof}
The key conditions to control are \eqref{wpsph1}--\eqref{wpsph4}.
Applying the second BGG operator \eqref{BGGce^A} to \eqref{wpsph2} yields 
$0=\bigl(  w^{CR}\, \B_2(\Ph)_{ABRD} \bigr)_0$,
where  $\B_2$ is as in \eqref{BGGce_A2}.
Here one has to take into account that $w^{AB}$ satisfies \eqref{wpsph1} or, equivalently, \eqref{wopProl1}--\eqref{wopProl2}.
In the flat case, the condition \eqref{wpsph7}, which is a consequence of the key ones, means 
$w^{RS} \B_2(\Ph)_{ARSD} =0$.
Altogether, we have
\begin{align*}
w^{CR}\, \B_2(\Ph)_{ABRD}=0 .
\end{align*}
For generic $\B_2(\Ph)$, this implies $w^{AB}=0$, hence the conditions \eqref{wpsph2} and \eqref{wpsph4} reduces to 
\begin{align*}
\bigl( D_{A} D_{B} v^C + \Rho_{AB} v^C \bigr)_0 =0 , \quad
D_A \mathring{\ps} = 0 .
\end{align*}

Applying the second BGG operator \eqref{BGGce_A2} to \eqref{wpsph3} yields 
\begin{align*}
0 = -\mathcal{L}_v (\B_2(\Ph)) + \mathring{\ps} \B_2(\Ph) ,
\end{align*}
where we have used that $\B_2$ commutes with $\mc{L}_v$ and $\mathring{\ps}$ is constant.
Locally, by the exactness of the BGG sequence, the right-hand side of \eqref{wpsph3} is the image of the first BGG operator on a section of $\ce_A(2)$. 
This guarantees the existence of $\al_A\in\ce_A(2)$ satisfying \eqref{wpsph3} and all such sections are parametrized by solutions to the equation \eqref{flat-al}.
However, in the flat case, the solution space has dimension $\frac12n(n+1)$, see section \ref{A2}.
\end{proof}

\subsection{Special modification}
Here we assume that $\Ph_{AB} = D_{(A}\ph_{B)}$, for some $\ph_A \in \ce_A(2)$, i.e.\ $\Ph_{AB}$ is in the image of the first BGG operator \eqref{BGGce_A2}. 
As discussed at the beginning of section \ref{Einstein-spec}, this case reduces to the standard conformal extension whose infinitesimal symmetries are characterized in \cite[Theorem 3]{hsstz-walker}.
Again, we derive the characterization directly from Proposition \ref{downCKE}, observe the effect of $\Phi_{AB}$ and note that the space of conformal Killing fields of a modified conformal extension of the current type is an affine space over the vector space in the non-modified situation.

\begin{thm} \label{cKimage}
Let the standard PW metric be modified by the term of the form  $\Ph_{AB} = D_{(A} \ph_{B)}$, for some $\ph_A \in \ce_A(2)$.
There is a bijective correspondence between its conformal Killing fields and triples $w^{AB} \in \ce^{[AB]}(-2)$, $v^A\in\ce^A$ and $\al_A \in \ce_A(2)$ satisfying 
\begin{align}
& \big(D_C w^{AB}\big)_0=0, \quad w^{R(C} W_{R(A}{}^{D)}{}_{B)} =0 , \label{spec-w} \\
& \bigl( D_{(A} D_{B)} v^C + \Rho_{AB} v^C + v^S W_{S(A}{}^C{}_{B)} \bigr)_0 =0 , \label{spec-v} \\
& D_{(A} \al_{B)} =0 . \label{spec-al}
\end{align}
The dimension of the space of conformal Killing fields equals to $d_1+d_2+d_3+1$, where $d_1$, $d_2$ and $d_3$ is the dimension of the space of solutions to \eqref{spec-w}, \eqref{spec-v} and \eqref{spec-al}, respectively.
\end{thm}

\begin{proof}
A direct computation shows that the bilinear operator from \eqref{B'op} has the form
\begin{align*}
\mc{F}(w,\Ph)^C{}_{AB} =
\tfrac12 w^{RC}\bigl( D_{(A} D_{B)} \ph_R+ \Rho_{AB} \ph_R  - \Rho_{R(A} \ph_{B)} + W_{R(A}{}^S{}_{B)} \ph_S \bigr) + \nu^C D_{(A} \ph_{B)} ,
\end{align*}
where $\nu^A = \tfrac{1}{n-1} D_Rw^{RA}$ as above.
Its trace-free and trace part, which appear on the right-hand side of \eqref{wpsph2} and \eqref{wpsph4}, respectively, are
\begin{align*}
& \bigl( \mc{F}(w,\Ph)^C{}_{AB} \bigr)_0 = 
\tfrac12 \bigl( D_{(A} D_{B)} (w^{RC}\ph_R) + \Rho_{AB} (w^{RC}\ph_R) + (w^{RS}\ph_R) W_{S(A}{}^C{}_{B)} \bigr)_0, \\
& \mc{F}(w,\Ph)^R{}_{AR} = 2 D_A \bigl( w^{RS} D_R\ph_S - 2\nu^R\ph_R  \bigr).
\end{align*}
Thus, $v^A = -\tfrac12 w^{RA}\ph_R$ is a particular solution to \eqref{wpsph2} and all such solutions are parametrized by solutions to the equation \eqref{spec-v}.
Similarly, $\mathring{\ps} = \tfrac{4}{n+1} ( w^{RS}D_R\ph_S - 2\nu^R\ph_R)$ is a particular solution to \eqref{wpsph4} and all such solutions differ by an additive constant.

Using the previous displays, the right-hand side of \eqref{wpsph3} can be arranged as
\begin{align*}
-\tfrac12 D_{(A} (\mathcal{L}_v\ph)_{B)} - \tfrac12 \bigl( \mc{F}(w,\Ph)^R{}_{AB} \bigr) \ph_R
+\tfrac{1}{2} \ph_{(A} D_{B)} \mathring{\ps} + \tfrac12 \mathring{\ps} D_{(A}\ph_{B)}.
\end{align*}
The problematic term $\bigl( \mc{F}(w,\Ph)^R{}_{AB} \bigr) \ph_R$ actually equals to $D_{(A} \ga_{B)}$,
where $\ga_B =  \tfrac12 (w^{RS} D_{B} \ph_R + \nu^S \ph_{B} ) \ph_S$.
Thus, 
$\al_B = -\tfrac12 (\mathcal{L}_v\ph)_{B} - \tfrac12 \ga_B + \tfrac12 \mathring{\ps} \ph_B$ 
is a particular solution to \eqref{wpsph3} and all such solutions are parametrized by solutions to the equation \eqref{spec-al}.
Equations in \eqref{spec-w} are just the remaining conditions from Proposition \ref{downCKE}.
\end{proof}

\subsection{Dimension four} \label{Symm2}
Analogously to section \ref{Einstein-2}, we inspect specific features of the lowest dimensional case.
In particular, the projective flatness is controlled by the vanishing of Cotton tensor $Y_{CAB}$.
Indices in computations below are lowered and raised  via the projective volume form $\bep_{AB} \in \ce_{[AB]}(3)$ and its inverse $\bep^{AB} \in \ce^{[AB]}(-3)$, respectively.

As the first result we show that, for non-flat extensions, one of the four building blocks necessarily vanishes:

\begin{prop} \label{summaryCKEdim2}
For a non-flat modified conformal extension of a 2-dimensional projective structure, there is a bijective correspondence between its conformal Killing fields and triples $v^A \in \ce^A$, $\al_A \in \ce_{A}(2)$ and $\mathring{\ps} \in \mathbb{R}$ satisfying
\begin{align}
& \bigl( D_{A} D_{B} v^C + \Rho_{AB} v^C \bigr)_0 =0 , \label{dim2-v-proj}\\
& D_{(A}^{} \al_{B)} = -\tfrac12 (\mathcal{L}_v \Ph)_{AB} + \tfrac12 \mathring{\ps} \, \Ph_{AB} . \label{dim2-al}
\end{align}
\end{prop}
Recall that the condition \eqref{dim2-v-proj} indicates that $v^A$ is an infinitesimal projective symmetry.

\begin{proof}
Sections $w^{AB}\in\ce^{[AB]}(-2)$ satisfying \eqref{wpsph1} are identified with densities $w^{AB}\bep_{AB} \in\ce(1)$ satisfying the first BGG equation corresponding to \eqref{BGGce1}.
Its non-trivial solutions has to satisfy the condition \eqref{compBGGce1}, which implies  projective flatness in dimension 2.
In such case, the condition \eqref{wpsph7} reads as $w^{RS} \B_2(\Ph)_{ARSD} =0$, which implies $\B_2(\Ph)=0$.
By Theorem \ref{relateFlat}, the conformal extension would be flat, which contradicts our assumption.

Thus, $w^{AB}$ has to vanish identically and the key conditions from Proposition \ref{downCKE} simplify to those displayed above.
\end{proof}

As above, the composition of an appropriate relation with the corresponding BGG operator may provide  extra information.
Let us continue in the setting of the previous Proposition and apply the second BGG operator \eqref{BGGce_A2} to \eqref{dim2-al}.
This yields
\begin{align*}
\al_{[C} Y_{D]AB} = \mc{L}_v(\B_2(\Ph))_{ABCD} -\mathring{\ps}(\B_2(\Ph))_{ABCD} ,
\end{align*}
where we have used \eqref{compBGGce_A2}, the fact that $\B_2$ commutes with $\mc{L}_v$ and that $\mathring{\ps}$ is constant.
With the associated quantities  $(\star Y)^A = Y_{CDE} \bep^{AC}\bep^{DE} \in \ce^A(-6)$ and $\star \B_2(\Ph) = \B_2(\Ph)_{ABCD} \bep^{AB}\bep^{CD} \in \ce(-4)$, the previous display is rewritten as
\begin{equation} \label{SymmComp}
\al_R (\star Y)^R = \mc{L}_v(\star \B_2(\Ph)) -\mathring{\ps}(\star \B_2(\Ph)) .
\end{equation}
Off its zero, the density $\star \B_2(\Ph) \in \ce(-4)$ plays a role of projective scale, hence fixes a unique affine connection from the projective class.

Compared to the context of Theorem \ref{eAdim2}, there are currently more terms entering the game and a complete discussion seems to be rather complicated.
Restricting to the case of non-flat extensions of flat projective structures, we conclude with the following specification of Theorem \ref{cKprojflat}:

\begin{thm} \label{cKf-dim2}
For a non-flat modified conformal extension of a flat 2-dimensional projective structure, let $D_A$ be the affine connection corresponding to the scale $\star \B_2(\Ph)$ and $R_{AB}{}^C{}_{D}$ be its curvature tensor.
Locally, off the zero set of $\star \B_2(\Ph)$, there is a bijective correspondence between the conformal Killing fields and pairs $v^A \in \ce^A$ and $\al_A \in \ce_{A}(2)$ satisfying
\begin{align}
& D_A D_B v^C + v^S R_{SA}{}^C{}_{B} =0 , \label{dim2-v-af} \\
& D_{(A} \al_{B)} =0 . \label{dim2-al-kill}
\end{align}
The dimension of the space of conformal Killing fields equals to $d+3$, where $d$ is the dimension of the space of solutions to \eqref{dim2-v-af}.
\end{thm}
Recall that the flatness of induced conformal structure is controlled by the vanishing of $\star\B_2(\Ph)$.
The condition \eqref{dim2-v-af} indicates that $v^A$ is an infinitesimal affine symmetry of $D_A$, see appendix~\ref{app-B}.

\begin{proof}
The scale $\star\B_2(\Ph)$ is parallel with respect to $D_A$, hence
\begin{align*}
\mc{L}_v(\star\B_2(\Ph)) = \tfrac43(D_R v^R)(\star\B_2(\Ph)) ,
\end{align*}
where the (otherwise unimportant) coefficient on the right-hand side reflects the conventions from \eqref{proj-w}.
With this observation we easily compare the actual conditions with those of Theorem \ref{cKprojflat}:

Let $v^A \in \ce^A$ be an infinitesimal projective symmetry satisfying \eqref{flat-lie} for some $\mathring{\ps}\in\mbb{R}$.
With the current reformulations, that condition means $D_R v^R =\frac34 \mathring{\ps}$. 
In particular, $D_R v^R$ is constant.
Hence, by Proposition \ref{projIS-affIS},  $v^A$ is an infinitesimal affine symmetry of $D_A$.
Conversely, 
let $v^A \in \ce^A$ be an infinitesimal affine symmetry of $D_A$.
Then $v^A$ is an infinitesimal projective symmetry and, by Proposition \ref{projIS-affIS}, $D_R v^R$ is constant.
Hence, for $\mathring{\ps}=\frac43 D_R v^R$, the condition  \eqref{flat-lie} is satisfied.
The rest is clear.
\end{proof}

By the local exactness of BGG sequences in the flat case, every projective scale is locally of the form $\star \B_2(\Ph)$, for some $\Ph$. 
It then follows from Theorem \ref{cKf-dim2} and Proposition \ref{aff-symm-2} that the corresponding modified conformal extension of a flat projective structure has the dimension of the symmetry algebra limited as follows:

\begin{thm} \label{cKdim2}
For non-flat modified conformal extensions of flat 2-dimensional projective structures, 
the Lie algebra of infinitesimal conformal symmetries can be of any dimension from 3 to 9, except for 8. 
\end{thm}

The flat 4-dimensional conformal structure, which is the standard conformal extension of flat projective structure, has maximal symmetry algebra, whose dimension is 15.
It turns out, that the submaximal dimension of the symmetry algebra of 4-dimensional conformal structures of split signature is 9, see \cite[section 5.1]{Kruglikov2017}. 
This case is discussed in more detail below.

\subsection{Submaximal example} \label{submax}

In this section, we illustrate some of the previously obtained general results in a very concrete setting.
For this purpose, we take a non-flat modified conformal extension with non-trivial algebra of infinitesimal symmetries and show how this (potentially complicated) structure can be assorted in simpler underlying projective terms.
Interestingly, examples of metrics of split signature whose algebra of conformal Killing fields has submaximal dimension are modified PW metrics, see \cite[section 5.1]{Kruglikov2017}.
The richest structure appears in dimension four, which is the case we discuss here in detail.

Let us take the metric
\begin{align}
\ol g = \d x^1\odot\d p_1 + \d x^2\odot\d p_2 + (x^2)^2 \d x^1\odot\d x^1 ,
\label{pp}
\end{align}
i.e. the modification with $\Ph=(x^2)^2 \d x^1\odot\d x^1$ of the flat PW metric.
According to Proposition \ref{relateW}, the Weyl curvature of $\ol{g}$ is 
\begin{align*}
\ol{W} =-2\B_2(\Ph) =-4 \d x^1\wedge\d x^2\odot\d x^1\wedge\d x^2 ,
\end{align*}
hence the conformal extension is non-flat.
According to Theorem \ref{cKf-dim2}, conformal Killing fields of $\ol{g}$ are parametrized by infinitesimal affine symmetries of the flat affine connection and projective Killing forms, which form  spaces of dimension 6 and 3, respectively.
Altogether, the space of conformal Killing fields has dimension 9.

To be more precise, we describe both the underlying sources and their corresponding lifts.
The infinitesimal affine symmetries and the projective Killing forms have the form, respectively,
\begin{align*}
& v = (c_1 +c_3 x^1 +c_5 x^2)\del_{x^1} +(c_2 +c_4 x^1 +c_6 x^2)\del_{x^2} , \\
& \al_k = (c_7+c_9 x^2)\d x^1 + (c_8-c_9 x^1)\d x^2 .
\end{align*}
where $c_1, \dots, c_9$ are arbitrary constants.
A particular solution to the key equation \eqref{wpsph3} is 
\begin{align*}
\al_p = \left(-c_2 x^1x^2 -\tfrac12c_4 (x^1)^2x^2 -\tfrac13c_5(x^2)^3\right)\d x^1 + \left(\tfrac12c_2(x^1)^2 +\tfrac16c_4(x^1)^3\right)\d x^2 ,
\end{align*}
for the constant $\mathring{\ps}=\frac43(c_3+c_6)$.
Thus, conformal Killing fields of $\ol{g}$ are uniquely given by the quadruples $w$, $v$, $\al$ and $\mathring{\ps}$, where  $w=0$ and $\al=\al_p+\al_k$.
According to Theorem \ref{summaryCKE}, they are described by the following lifts:
\begin{align*}
& \ol v_0 = \left(c_1 +c_3 x^1 +c_5 x^2\right)\del_{x^1} +\left(c_2 +c_4 x^1 +c_6 x^2\right)\del_{x^2} + \\
& \hskip6em + \left((c_3+2c_6) p_1 -c_4 p_2\right)\del_{p_1} + \left(-c_5 p_1 +(2c_3+c_6) p_2 \right)\del_{p_2} , \\
& \ol v_- = \left(c_7+c_9 x^2-c_2 x^1x^2 -\tfrac12c_4 (x^1)^2x^2 -\tfrac13c_5(x^2)^3\right)\del_{p_1} + \\
& \hskip10em + \left(c_8-c_9 x^1+\tfrac12c_2(x^1)^2 +\tfrac16c_4(x^1)^3\right)\del_{p_2} .
\end{align*}
All in all, conformal Killing fields of the modified PW metric \eqref{pp} have the form $\ol v=\ol v_0+\ol v_-$ with the summands as above.

Also, the structure of the Lie algebra of conformal Killing fields is given by the underlying data, typically in a rather intricate way.
To decipher the current example, let us denote the algebra of infinitesimal affine and conformal symmetries by $\g$ and $\ol\g$, respectively.
Let us further denote by $e_i$ and $\ol e_i$ the generator of $\g$ and $\ol\g$ corresponding to the coefficient $c_i$, where $i$ runs from 1 to 6 and 9, respectively.
Of course, $\g=\g_0\oplus\g_1$, where the reductive and the nilpotent subalgebra is, respectively, 
\begin{gather*}
\g_0 =  \langle e_4, e_3-e_6, e_5 \rangle \oplus \langle e_3+e_6 \rangle, \\
\g_1 =\langle e_1, e_2 \rangle.
\end{gather*}
The indicated decomposition of $\g_0$ exhibits its simple part, which is $\mf{sl}(2,\mbb{R})$, and the center.

It turns out that $\ol\g$ is a graded Lie algebra, $\ol\g=\ol\g_0\oplus\ol\g_1\oplus\ol\g_2\oplus\ol\g_3$, where
\begin{gather*}
\ol\g_0 =\langle \ol{e}_4, \ol{e}_3-\ol{e}_6, \ol{e}_5 \rangle \oplus \langle \ol{e}_3+\ol{e}_6 \rangle, \\
\ol\g_1 =\langle \ol{e}_1, \ol{e}_2 \rangle, \quad
\ol\g_2 =\langle \ol{e}_9 \rangle, \quad
\ol\g_3 =\langle \ol{e}_7, \ol{e}_8 \rangle .
\end{gather*}
The subalgebra $\ol\g_0$ is isomorphic to $\g_0$, the effect of the modification is visible in the nilpotent part of $\ol\g$.
Nonetheless, the action of $\ol\g_0$ on $\ol\g_1$, respectively on  $\ol\g_2\oplus\ol\g_3$, corresponds to the action of $\g_0$ on $\g_1$, respectively on projective Killing forms.
The resulting refined structure on the conformal side is a priori not obvious and emerges as a consequence of the very specific modification.
Moreover, note that $\ol\g$ is isomorphic to a parabolic subalgebra of the split real form of $14$-dimensional exceptional simple Lie algebra $\g_2$, namely, to the one corresponding to the short root.

\begin{remas}
The appearance of a parabolic subalgebra of $\g_2$ as the Lie algebra of conformal infinitesimal symmetries of the metric \eqref{pp} has the following geometric explanation.
There is a natural (twistor) distribution of rank 2 on the bundle of self-dual null-planes of a 4-dimensional split-signature conformal manifold.
Locally, this is a $(2,3,5)$ distribution if and only if the self-dual part of the Weyl tensor is non-trivial, see \cite{An2012}.
In this correspondence, infinitesimal symmetries of the conformal structure lift to  infinitesimal symmetries of the twistor distribution, and thus the symmetry algebra of the former structure is identified with a subalgebra of the symmetry algebra of the latter one.
It is a classical result by Cartan, 
that the maximal symmetry algebra of $(2,3,5)$ distributions is the exceptional simple Lie algebra $\g_2$ and the submaximal one has dimension 7.
Now, the self-dual part of the Weyl tensor of \eqref{pp} is non-trivial and its algebra of conformal symmetries has dimension 9. 
It follows that the Lie algebra of infinitesimal conformal symmetries of \eqref{pp}  has to be a subalgebra of $\g_2$.
Finally, note that there are just two 9-dimensional subalgebras in $\g_2$, namely, the two maximal parabolic ones.

The modified PW metric \eqref{pp} is also a pp-wave, symmetric and Einstein metric.
All Einstein metrics in the conformal class  of \eqref{pp} correspond to the scales of the form $\si=c_1x_1+c_2x_2+c_3$, for $c_1,c_2,c_3\in\mbb R$. 
They form a space of dimension 3 and provide a realization of one of the possibilities listed in Theorem \ref{eAdim2}.
\end{remas}

\section{Ambient metric and $Q$-curvature} \label{Ambient}

There is a natural geometric construction of the Fefferman--Graham ambient metric for modified conformal PW metrics.
Compared to the non-modified situation studied in \cite{hsstz-ambient}, the modification tensor causes no serious complication.
In this context, partial observations with more general modifications were done, which we comment at the end of section.

The general setting of the Fefferman--Graham ambient construction is as follows, see \cite{Fefferman2012}.
Let a conformal class on a smooth manifold $\wt{M}$ be represented by a metric $g$.
The ambient manifold $\mb{M}$ is locally described by two more coordinates, say $t\in\mbb{R}_+$, which parametrizes metrics in the conformal class, and $\rho\in\mbb{R}$, which represents the essentially new dimension.
The normal form for the ambient metric on $\mb{M}$ is
\begin{align} \label{ambient-normal}
\mb{g}=2\rho \d t\odot \d t + 2t \d t\odot \d\rho +t^2 g(\rho) ,
\end{align}
where $g(\rho)$ is a smooth 1-parameter family of metrics on $\wt{M}$ with $g(0)=g$ being a representative metric in the conformal class.
On of the key features of such metrics is that they are homogeneous of degree 2 with respect to $t$.
The \emph{Fefferman--Graham ambient metric} is a metric of the form \eqref{ambient-normal} that is Ricci-flat.
The existence of such metrics is a subtle question. 
For a given $g$, the metric \eqref{ambient-normal} is typically constructed iteratively by the Taylor expansion of $g(\rho)$ in $\rho$ so that the Ricci-flatness condition is controlled asymptotically for $\rho=0$.
In even dimension, the construction is obstructed in finite order by the so-called Fefferman–Graham tensor, a conformally invariant tensor, which is the Bach tensor in dimension 4.

For modified conformal extensions of projective structures, the ambient metric exists and has the simplest conceivable form.
Namely, the term $g(\rho)$ is linear in $\rho$ and $\mb{g}$ is Ricci-flat globally, not only asymptotically:

\begin{thm} \label{Theorem_ambient}
Let $\ol{g}$ be a modified PW metric with the Schouten tensor $\ol\Rho$ and let $t$, $\rho$ be the additional coordinates on the ambient manifold.
Then 
\begin{align} \label{ambient-local}
\mb{g}=2\rho \d t\odot \d t + 2t \d t\odot \d\rho +t^2 \Big( \ol{g} +2\rho\,\ol\Rho \Big)
\end{align}
is a globally Ricci-flat Fefferman--Graham ambient metric of the conformal class of $\ol{g}$.
\end{thm}

The Ricci, respectively Schouten, tensor of a modified PW metric is the pullback of the Ricci, respectively Schouten, tensor of the underlying affine connection, independently of the modification tensor $\Phi$, see section \ref{Curva}.
Hence, in local coordinates $(x^A,p_A)$ as above, the Schouten tensor of $\ol{g}$ has the form $\ol\Rho=\frac{1}{n-1}\,\Ric_{AB}\d x^A\odot\d x^B$, where $\Ric_{AB}$ is the Ricci tensor of the underlying affine connection and $n$ is the dimension of the corresponding manifold.

\begin{proof}
The normal form of the metric \eqref{ambient-local} is obvious, thus it is enough to check the Ricci-flatness.
This can be done directly, using \cite[formula (3.17)]{Fefferman2012} and the just mentioned properties of the Ricci tensor of $\ol{g}$.
Alternatively, the metric $\mb{g}$ can be obtained as the modified PW metric of the Thomas ambient connection associated to the initial projective structure.
The Ricci-flatness of $\mb{g}$ then follows from the Ricci-flatness of the Thomas ambient connection.
Details in this picture are as follows:

Let $\mc{C}$ and $\nabla$ be the Thomas ambient cone and connection, respectively, associated to the underlying projective structure.
The modification tensor $\Phi\in\ce_{(AB)}(2)$ lifts to a symmetric 2-tensor $\wh\Phi$ on $\mc{C}$ which is homogeneous of degree 2 with respect to the standard $\mbb{R}_+$-action on $\mc{C}$.
Finally, the $\hat\Phi$-modified PW metric of $\nabla$ is identified with \eqref{ambient-local} by the same coordinate transformation as in the proof of \cite[Theorem 2]{hsstz-ambient}.
\end{proof}

Another related notion is the \emph{$Q$-curvature}.
Although it is associated to a particular metric rather than to the conformal class, it is an important quantity in conformal geometry, see \cite{branson-functional}, \cite{chang-eastwood-Q}, \cite{fefferman-hirachi}.
The explicit and simple form of the Fefferman--Graham ambient metric allows the $Q$-curvature to be computed, which is generally a rather difficult task.

To do so, one analyzes $\mb\Delta^n\log(t)$, where $\mb\Delta$ is the ambient Laplacian associated to \eqref{ambient-local} and $t$ as above.
In the non-modified case, it easily follows that $\mb\Delta\log(t)=0$, hence the $Q$-curvature of any PW metric vanishes, see \cite[Theorem 3]{hsstz-ambient}.
Checking details in that reference, it follows that the modification tensor $\Phi$ does not play any role.
Hence we have the following generalization:

\begin{thm} \label{Q-curvature}
The $Q$-curvature of any modified PW metric vanishes.
\end{thm}

\begin{remas}
Modified PW metrics fit into the class of so-called \emph{null Ricci Walker metrics} discussed in \cite{Anderson2016}, for which the existence of explicit ambient metrics has been shown in some interesting cases.
Modified PW metrics also form a subclass of more general modifications studied in \cite{Gilkey2009a}, which include, in particular, self-dual Walker metrics in dimension 4.
The Bach tensor of a self-dual conformal metric, i.e. the Fefferman--Graham obstruction tensor, vanishes identically.
While no geometric construction of the associated ambient metric is known for this class in general, an explicit formula for specific cases can be obtained.

For the sake of illustration, we allow extra modifications parametrized by smooth functions $\al\in\ce$ on the underlying projective manifold and determined by the distinguished vertical field $k^a\in\wt\ce^a$ from Remark \ref{rem-char}, namely, the homothety of the standard PW metric $\wt{g}_{ab}$.
Thus, we allow the generalization of  \eqref{eq-modPW} so that
\begin{align}
\ol{\ol{g}}_{ab}:=\wt{g}_{ab} +\Phi_{ab} +\al\, k_{(a} k_{b)} .
\label{extraPW}
\end{align}
Particular modifications of the form \eqref{extraPW} appear also in \cite{Dunajski2018}.
In the case when $\wt{g}$ is the flat PW metric, $\Phi=0$ and $\al=1$ we have the \emph{para Fubini--Study metric}, cf.  \cite{Bor2015}, \cite{Dunajski2019a}.
\end{remas}

\begin{prop}
Let $\ol{\ol{g}}$ be the extra modified PW metric \eqref{extraPW}, for some $\al\in\ce$, on a 4-dimensional manifold with the Schouten tensor $\ol{\ol\Rho}$ and let $t$, $\rho$ be the additional coordinates on the ambient manifold.
Then 
\begin{align} \label{ambient-extra}
\mb{g}=2\rho \d t\odot \d t + 2t \d t\odot \d\rho 
+t^2 \left( \ol{\ol{g}} +2\rho\,\ol{\ol\Rho} +\rho^2 \big( \al^2\ol{\ol{g}} +2\al\,(\d\alpha)\odot k \big) \right)
\end{align}
is a globally Ricci-flat Fefferman--Graham ambient metric of the conformal class of $\ol{\ol{g}}$.
Moreover, the $Q$-curvature of the metric $\ol{\ol{g}}$ is $Q=-56\al^2$.
\end{prop}

Again, the metric \eqref{ambient-extra} is in the normal form, so it suffices to check its Ricci-flatness.
This was done directly using the computational systems \textsc{Mathematica}.
It follows that the expression \eqref{ambient-extra} is not unique: 
adding an arbitrary constant multiple of $(\wt{\Delta}\al)\,\ol{\ol{g}} -2\al\,(\d\al)\odot k$, 
where $\wt\Delta$ is the Laplacian of $\wt{g}$, into the inner parentheses provides also a Ricci-flat solution.

\appendix

\section{Projective BGG operators} \label{app-A}

The BGG theory is a conceptual framework comprehending many problems concerning invariant operators, respectively equations, in parabolic geometries.
For a general introduction, we refer to seminal papers \cite{Cap2001} and \cite{Calderbank2001}.
BGG operators appear in sequences, in which the first operator is the most frequent in applications.
Solutions to the corresponding equations often have significant geometrical meaning.
Prominent instances in conformal geometry are related to almost Einstein scales and infinitesimal symmetries.
These problems are studied in sections \ref{Einstein}--\ref{cKspecial}, where the corresponding overdetermined equations are reduced to systems of projectively invariant equations among which projective first BGG operators play an essential role.

Each BGG sequence is determined by a concrete tractor bundle and the related exterior covariant differential, from which an explicit description of the operators can be deduced.
On the underlying level, the type of sequence is encoded in the source bundle of the first operator.
In the locally flat case, BGG sequences form complexes, which are locally exact. 
Moreover, solutions to the first BGG equation form a vector space whose dimension equals to the rank of the background tractor bundle.

For purposes of this article, we need first two operators of several projective BGG sequences.
We denote these operators as $\B_1$ and $\B_2$, keeping in mind they are always related to the actual type of the sequence.
Each sequence corresponds to certain tensor product of the standard tractor bundle $\mc{T}$ and its dual $\mc{T}^*$.
The standard tractor bundle is associated to the standard action of the group $SL(n+1,\mbb{R})$, the Lie group of projective symmetries of the (oriented) projective sphere of dimension $n$.
In particular, the rank of $\mc{T}$ is $n+1$.

The common pattern of the following subsections is as follows.
Firstly, we specify the type of sequence and explicit formulas for first two operators.
Secondly, composing these two operators, we get an integrability condition for the existence of solutions to the first BGG equation.
Thirdly, we count the dimension of the solution space to that equation in the locally flat case.
Some operators depend on the dimension $n$ of the projective manifold.
This is caused by natural isomorphisms in low dimensions that are, on the underlying level, provided by the projective volume form
$\bep_{A_1 \ldots A_n} \in \ce_{[A_1 \ldots A_n]}(n+1)$, respectively its dual 
$\bep^{A_1 \ldots A_n} \in \ce^{[A_1 \ldots A_n]}(-n-1)$.

\def\barrow#1{\mathop{\longrightarrow}^{\B_{#1}}}
\def\rank{\operatorname{rank}}

\subsection{Sequence for $\mc{T}^*$} \label{A1}
The  corresponding BGG sequence starts as follows:
\begin{align*}
\ce(1) \barrow1 \ce_{(AB)}(1) \barrow2 \ce_{[CA]B}(1) \longrightarrow \cdots  
\end{align*}
where 
\begin{align}
\B_1(\ta)_{AB} = (D_A D_B + \Rho_{AB}) \ta, \qquad 
\B_2(T)_{CAB} =  D_{[C} T_{A]B} .
\label{BGGce1}
\end{align}
The composition yields
\begin{equation} \label{compBGGce1}
(\B_2 \circ \B_1)(\ta)_{CAB} = - \tfrac12 \bigl( W_{AB}{}^R{}_C D_R \ta - Y_{CAB} \ta).
\end{equation}
For $n=2$, the vanishing of \eqref{compBGGce1} means $Y_{CAB}=0$, i.e. the local flatness.
In locally flat case, the dimension of the solution space of $\B_1$ equals to 
$$
\rank\mc{T}^* = n+1.
$$

\subsection{Sequence for $\wedge^2\mc{T}^*$} \label{A2}
The corresponding BGG sequence starts as follows:
\begin{align*}
\ce_B(2) \barrow1 \ce_{(AB)}(2) \barrow2 \ce_{[AB][CD]}(2) \longrightarrow \cdots  
\end{align*}
where 
\begin{align}
\B_1(\ph)_{AB} = D_{(A}\ph_{B)}, \qquad
\B_2(\Phi)_{ABCD} = \Proj_\boxplus \bigl( D_A D_C \Phi_{BD} + \Rho_{AC} \Phi_{BD} \bigr) .
\label{BGGce_A2}
\end{align}
Here $\Proj_\boxplus$ denotes the projection to the subspace of the `window' symmetry corresponding to the indicated Young tableau, i.e. the subspace of $\ce_{[AB][CD]}(2)$ such that the skew-symmetrization over any triple of indices vanishes. 
Otherwise put, a short computation reveals that
\begin{equation} \label{B2expl}
\begin{split}
\B_2(\Phi)_{ABCD} = \Proj_{[AB][CD]} \bigl( D_A D_C \Phi_{BD} &+ \Rho_{AC}\Phi_{BD} \, +\\
&+ \tfrac14 W_{AB}{}^R{}_C \Ph_{DR} - \tfrac14 W_{CD}{}^R{}_A \Ph_{BR} \bigr) ,
\end{split}
\end{equation}
where $\Proj_{[\cdot\,\cdot][\cdot\,\cdot]}$ denotes the skew-symmetrization over the embraced indices.
The composition yields
\begin{align}
\begin{split} \label{compBGGce_A2}
(\B_2 \circ \B_1)(\ph)_{ABCD}= \Proj_{[AB][CD]} \bigl( 
& -\tfrac12 W_{AB}{}^R{}_C D_{[R} \ph_{D]} -\tfrac12 W_{CD}{}^R{}_A D_{[R} \ph_{B]} \, + \\
& + \tfrac14 W_{AB}{}^R{}_C D_{(D} \ph_{R)} - \tfrac14 W_{CD}{}^R{}_A D_{(B} \ph_{R)} \, - \\
& \hskip4em -\tfrac12 ( D_A W_{CD}{}^R{}_B )\ph_R - \tfrac12 \ph_C Y_{DAB} 
\bigr).
\end{split}
\end{align}
For $n=2$, the vanishing of \eqref{compBGGce_A2} means $\ph_{[C} Y_{D]AB}=0$.
This condition can be expressed so that $\ph_A = f (\star Y)_A$, where $(\star Y)_A := Y_{ABC}\bep^{BC}$, for some $f \in \ce(5)$.
In locally flat case, the dimension of the solution space of $\B_1$ equals to 
$$
\rank\wedge^2\mc{T}^* = \frac12n(n+1).
$$

\subsection{Sequence for $\mc{T}$} \label{A3}
The corresponding BGG sequence starts with $\ce^A(-1)$.
For $n=2$, we have the identification $\ce^A(-1) \cong \ce_A(2)$, hence the sequence coincides with the one in section \ref{A2}.
For $n\ge 3$, the sequence is as follows:
\begin{align*}
\ce^B(-1) \barrow1 \bigl( \ce_A{}^B(-1) \bigr)_0 \barrow2 \bigl( \ce_{[CA]}{}^B (-1)) \bigr)_0 \longrightarrow \cdots  
\end{align*} 
where
\begin{align}
\begin{split} \label{BGGce^A(-1)}
\B_1(\xi)_A{}^B = \bigl( D_A \xi^B \bigr)_0, \qquad   
\B_2(\Xi)_{CA}{}^B =  \bigl( D_{[C} \Xi_{A]}{}^B \bigr)_0 .
\end{split}
\end{align}
The composition yields
\begin{equation} \label{compBGGce^A(-1)}
(\B_2 \circ \B_1)(\xi)_{CA}{}^B = \tfrac12 W_{CA}{}^B{}_R \xi^R.
\end{equation}
In locally flat case, the dimension of the solution space of $\B_1$ equals to
$$
\rank\mc{T} = n+1.
$$

\subsection{Sequence for $\wedge^2\mc{T}$} \label{A4}
The corresponding BGG sequence starts with $\ce^{[AB]}(-2)$.
For $n=2$, we have the identification $\ce^{[AB]}(-2) \cong \ce(1)$, hence the sequence coincides with the one in section \ref{A1}.
For $n=3$, we have the identification $\ce^{[AB]}(-2) \cong \ce_A(2)$, hence the sequence coincides with the one in section \ref{A2}.
For $n\ge 4$, the sequence is as follows:
\begin{align*}
\ce^{[AB]}(-2) \barrow1 \bigl( \ce_D{}^{[AB]}(-2) \bigr)_0 \barrow2 \bigl( \ce_{[CD]}{}^{[AB]} (-2) \bigr)_0 \longrightarrow \cdots  
\end{align*}
where
\begin{align}
\begin{split} \label{BGGce^AB(-2)}
\B_1(w)_D{}^{AB} = \bigl( D_D w^{AB} \bigr)_0, \qquad   
\B_2(V)_{CD}{}^{AB} =  \bigl( D_{[C} V_{D]}{}^{AB} \bigr)_0 .
\end{split}
\end{align}
The composition yields
\begin{equation} \label{compBGGce^AB(-2)}
(\B_2 \circ \B_1)(w)_{CD}{}^{AB} = -  \bigl( W_{CD}{}^{[A}{}_R w^{B]R}  \bigr)_0.
\end{equation}
Note that, for $n=3$, the first operator actually coincide with the one in \eqref{BGGce^AB(-2)} (whereas the second does not) and the expression \eqref{compBGGce^AB(-2)} vanishes identically for any $w^{AB}$.
In locally flat case, the dimension of the solution space of $\B_1$ equals to
$$
\rank\wedge^2\mc{T} = \frac12n(n+1).
$$

\def\mm#1{#1^\#}

\subsection{Sequence for $(\mc{T}^* \otimes \mc{T})_0$} \label{A5}
The corresponding BGG sequence starts as follows:
\begin{align*}
\ce^C \barrow1 \bigl( \ce_{(AB)}{}^C \bigr)_0 \barrow2 \bigl( \ce_{[DA]B}{}^C \bigr)_0 \longrightarrow \cdots  
\end{align*}
where 
\begin{align}
\begin{split} \label{BGGce^A}
\B_1(v)_{(AB)}{}^C = \bigl( D_A D_B v^C + \Rho_{AB}v^C  \bigr)_0, \qquad   
\B_2(V)_{DAB}{}^{C} =  \bigl( D_{[D}V_{A]B}{}^C \bigr)_0 .
\end{split}
\end{align}
In the context of projective infinitesimal symmetries, we refer also to the following modification of the first operator:
\begin{align}
\mm{\B}_1(v)_{AB}{}^C = \bigl( D_A D_B v^C + \Rho_{AB}v^C + v^R W_{R{(A}}{}^C{}_{B)} \bigr)_0,  
\label{BGGwt}
\end{align}
One can verify that both $\B_1(v)_{AB}{}^C$ and $\mm{\B}_1(v)_{AB}{}^C$ are, indeed, symmetric in lower indices.
In fact, the target space of $\B_2$ is the subspace of $\bigl( \ce_{[DA]B}{}^C \bigr)_0$ of such tensor fields that vanish upon skew-symmetrization over lower indices.
The composition yields
\begin{equation} \label{compBGGwt}
(\B_2 \circ \mm{\B}_1)(v)_{DAB}{}^C = \tfrac12 (\mathcal{L}_v W)_{DA}{}^C{}_B ,
\end{equation}
where $\mathcal{L}_v$ denotes the Lie derivative in the direction of the vector field $v^A$.
In locally flat case, the dimension of the solution space of $\B_1$ equals to
$$
\rank(\mc{T}^* \otimes \mc{T})_0 = n(n+2).
$$

\section{Projective and affine infinitesimal symmetries} \label{app-B}

Both parts of this section concern affine, respectively projective, infinitesimal symmetries of an affine connection.
The following results are relevant primarily for section \ref{Symm2}.

\subsection{Affine symmetries among projective ones}

Let $D_A$ be a special torsion-free affine connection and $R_{AB}{}^C{}_D$ be its curvature tensor.
Let $\Rho_{AB}$ and $W_{AB}{}^C{}_D$ be the projective Schouten and Weyl tensor of $D_A$, respectively.
These tensors are related by
\begin{align}   \label{formulaCprojective}
W\ind{_{AB}^C_D} =R\ind{_{AB}^C_D} +\Rho_{AD}\de\ind*{^C_{B}}- \Rho_{BD}\de\ind*{^C_{A}} .
\end{align}

It is a common knowledge 
that a vector field $v^A\in\ce^A$ is an affine infinitesimal symmetry of the affine connection $D_A$ if and only if it satisfies 
\begin{equation} \label{affIS}
D_A D_B v^C + v^S R_{SA}{}^C{}_{B} =0 
\end{equation}
and it is a projective infinitesimal symmetry of the projective structure $[D_A]$ if and only if it satisfies
\begin{equation} \label{projIS}
\bigl( D_A D_B v^C + \Rho_{AB}v^C + v^R W_{R(A}{}^C{}_{B)} \bigr)_0=0.
\end{equation}
Clearly, any affine infinitesimal symmetry is also the projective one. 
The opposite consideration is specified as follows:

\begin{prop} \label{projIS-affIS}
Let $v^A \in \ce^A$ be a projective infinitesimal symmetry of a projective structure $[D_A]$.
Then $v^A$ is an affine infinitesimal symmetry of a representative special affine connection $D_A$ if and only if the divergence $D_R v^R$ is constant, i.e., $D_A D_R v^R=0$.
\end{prop}

\begin{proof}
The `only if' direction of the equivalence is read directly from \eqref{affIS}.
For the `if' direction, we assume that $v^A$ satisfies \eqref{projIS} and $\ps := \tfrac1n D_Rv^R$ is a constant function and want to show that \eqref{affIS} holds.
Let us decompose the left-hand side of \eqref{affIS} into the skew-symmetric, the symmetric trace-free and the trace part, respectively.
The skew-symmetric part vanishes trivially.
The condition \eqref{projIS} implies that the symmetric trace-free part also vanishes.
It remains to show that traces vanish as well.
An easy computation shows that (any) trace is a nonzero constant multiple of $\Rho_{AR}v^R$.
I.e., \eqref{affIS} is satisfied if and only if $\Rho_{AR}v^R=0$.

The equation \eqref{projIS} is overdetermined and its prolongation was computed in \cite[section 6]{hsstz-walker}. 
Besides $v^A$ and $\ps$, two additional variables are needed to form a closed system, namely,
\begin{align}
\phi_A^B := D_Av^B - \de_A^B \ps, \qquad
\be_A := -\tfrac{1}{n+1}D_A D_Rv^R - \Rho_{AR}v^R .
\label{fuck-off}
\end{align}
Now, solutions to \eqref{projIS} are exactly the solutions to the prolonged system, which is displayed in \cite[formulas (71)]{hsstz-walker} and which we do not reproduce here.
By the assumption that $\ps$ is constant, the second quantity in \eqref{fuck-off} reads as  $\be_A = -\Rho_{AR}v^R$.
Substituting everything to the third equation of the prolonged system, one obtains $\Rho_{AR}v^R=0$ after some computation.
\end{proof}

\subsection{Affine symmetries in dimension 2}
The specification of dimensions of algebras of affine infinitesimal symmetries is a classical subject.
It is known that, for affine connections on 2-dimensional manifolds, all dimension less or equal to 6, except for 5, are possible, see \cite{Egorov1967} and references therein.
In fact, all these dimensions can be realized as follows:

\begin{prop} \label{aff-symm-2}
The Lie algebra of affine infinitesimal symmetries of an affine connection on a 2-dimensional manifold can be of any dimension less or equal to 6, except for 5.
Each of these cases can be realized by a special projectively flat affine connection.
\end{prop}

\begin{proof}
Let $(x^1,x^2)$ be a local coordinates and let us consider the family of affine connections obtained from the flat affine connection by the projective change \eqref{eq-Dhat-D}, where
\begin{align*}
\Ups = (a_2 (x^1)^2+a_1 x^1 + a_0) \d x^1 + (b_1 x^2+b_0) \d x^2 ,
\end{align*}
where $a_0,a_1,a_2,b_0,b_1$ are real parameters.
I.e., we consider the family of connections whose Christoffel symbols are
\begin{gather*}
\tfrac{1}{2}\Gamma_1{}^1{}_1=\Gamma_1{}^2{}_2=\Gamma_2{}^2{}_1 = a_2 (x^1)^2+a_1 x^1 + a_0, \\  
\tfrac{1}{2}\Gamma_2{}^2{}_2=\Gamma_1{}^1{}_2=\Gamma_2{}^1{}_1 = b_1 x^2+b_0, \qquad
\Gamma_2{}^1{}_2=\Gamma_1{}^2{}_1 = 0.
\end{gather*}
For particular values of parameters, we achieve all the possibilities as follows:
\begin{enumerate}[(1)]
\item[(0)]
Generic values of parameters provide no affine symmetry.
\item[(1)]
For $a_2=0$, the symmetry algebra has dimension 1.
\item[(2)]
For $a_2=b_1=0$, the symmetry algebra has dimension 2.
\item[(3)]
For $a_2=b_1=b_0=0$, the symmetry algebra has dimension 3.
\item[(4)]
For $a_2=b_1=b_0=a_1=0$, the symmetry algebra has dimension 4.
\item[(6)]
For all parameters vanishing, i.e. for the flat connection, the symmetry algebra has dimension 6.
\end{enumerate}
The symmetry of the Ricci tensor was checked and the affine symmetry algebra was generated using the computational system \textsc{Maple}.
\end{proof}


\end{document}